\documentclass[12pt]{amsart}
\usepackage{amsfonts, amssymb, amscd, latexsym, graphicx, psfrag, color, float, enumitem}
\usepackage[all]{xy}
\usepackage{mathrsfs}

\usepackage{tikz}

\usepackage[hyphens]{url}

\usepackage{todonotes}
\usepackage{hyperref}

\usepackage[hyphenbreaks]{breakurl}

\usepackage{tkz-euclide}
\usepackage{caption}

\newtheorem{dummy}{dummy}[section]
\newtheorem{lemma}[dummy]{Lemma}
\newtheorem{theorem}[dummy]{Theorem}
\newtheorem{innercustomthm}{Theorem}
\newenvironment{customthm}[1]
{\renewcommand\theinnercustomthm{#1}\innercustomthm}
  {\endinnercustomthm}

\newtheorem{corollary}[dummy]{Corollary}
\newtheorem{proposition}[dummy]{Proposition}
\theoremstyle{definition}
\newtheorem{definition}[dummy]{Definition}
\newtheorem*{definition*}{Definition}
\newtheorem{example}[dummy]{Example}

\newtheorem{remark}[dummy]{Remark}

\newtheorem{question}[dummy]{Question}

\newcommand{\bA}{\mathbb{A}}
\newcommand{\bC}{\mathbb{C}}

\newcommand{\cE}{\mathcal{E}}

\newcommand{\cK}{\mathcal{K}}

\newcommand{\cM}{\mathcal{M}}
\newcommand{\cO}{\mathcal{O}}

\newcommand{\cT}{\mathcal{T}}

\newcommand{\cX}{\mathcal{X}}

\newcommand{\Hom}{Hom}

\newcommand{\Perf}{\mathrm{Perf}}
\newcommand{\Ell}{\mathcal{E} ll}
\newcommand{\Spec}{\mathrm{Spec}}

 \newcommand{\twocell}[1]{\ar@{}[#1]^(.30){}="a"^(.70){}="b" \ar@{=>} "a";"b"}
 \newcommand{\ocell}[1]{\ar@{}[#1]^(.30){}="a"^(.70){}="b" \ar@{=} "a";"b"}
\newcommand{\Qcoh}{\mathrm{Qcoh}}

\begin{document}

\author[Scherotzke]{Sarah Scherotzke}
\address{Sarah Scherotzke, 
Department of Mathematics\\
Universite du Lexembourg\\ 
6 Av. de la Fonte\\ 4364 Esch-sur-Alzette\\ Luxembourg}
\email{\href{mailto:sarah.scherotzke@uni.lu}{sarah.scherotzke@uni.lu}}

\author[Sibilla]{Nicol\`o Sibilla}
\address{Nicol\`o Sibilla, SISSA\\ 
Via Bonomea 265\\ 34136 Trieste TS\\
Italy}
\email{\href{mailto:nsibilla@sissa.it}{nsibilla@sissa.it}}

\title[Equivariant elliptic cohomology, toric varieties, and derived equivalences]{Equivariant elliptic cohomology, toric varieties, and derived equivalences}

%\subjclass[2010]{14F05, 19E08}
%\keywords{Semi-orthogonal decompositions, Kummer flat K-theory}

\begin{abstract} 
In this article we study the equivariant elliptic cohomology of complex toric varieties. We prove a  partial reconstruction theorem  showing that equivariant elliptic cohomology encodes considerable non-trivial information on the equivariant 1-skeleton of a toric variety $X$ (although it stops short of being a complete invariant of its GKM graphs). Elliptic cohomology is supposed to encode higher categorical geometric data, and proposals have been made linking elliptic cocycles to categorified bundles. In particular, contrary to ordinary cohomology and K-theory,  elliptic cohomology is expected not to be a derived invariant of algebraic varieties. Our second main result is to verify this prediction by showing that there exist  pairs of equivariantly  derived equivalent toric varieties with non-isomorphic equivariant elliptic cohomology.
 \end{abstract}

\maketitle

%\tableofcontents

\section{Introduction} 
The first indications of the existence of elliptic cohomology theories came from work of Landweber--Stong and Ochanine on elliptic genera. Witten's construction of a  kind of universal  elliptic genus (i.e. the  Witten genus) revealed that these objects had deep connections  with quantum field theory, and gave further impetus to this area. Today elliptic cohomology is a  central objects in  homotopy theory. We refer the reader to  \cite{lurie2009survey} for a beautiful  overview of   this area of intense current research.    One of the main open  problems in elliptic cohomology is to find  % whether there is 
a meaningful geometric construction of elliptic  cocycles. Cocycles  in complex K-theory are represented by vector bundles but no such description of elliptic cocycles has been fully established, although there have been   important contributions in this direction starting with ideas of Segal, and the Stolz--Teichner program.

A key insight from chromatic homotopy theory is that increasing the chromatic level ought to involve an increase in categorical complexity. Elliptic cohomology is of chromatic level two; this has led, for instance, to the idea  % it  has been suggested 
that cocycles in elliptic cohomology should be related to 2-bundles, i.e. bundles of categories over a space \cite{baas2004elliptic}. 
In this article we make no step forward on the difficult question of finding a geometric meaning for elliptic cocycles. However  our findings corroborate the idea that elliptic cohomology detects a   layer of the geometry of a space  which is not captured by 1-categorical  data. To show this, we prove a general structure result on the equivariant elliptic cohomology of toric varieties which is of independent interest.%, and describes to what extent  the equivariant 1-skeleton of a space can be reconstructed purely from its  elliptic cohomology.

More precisely, our goal in this article  is  twofold: 
\begin{itemize}
\item We prove that the equivariant elliptic cohomology of a toric variety encodes highly non-trivial information  on  the geometry of the variety $X$ and 
 on the $T$-action (see Theorem \ref{mainA} of this Introduction). Namely, it allows us to recover a substantial part of the GKM graph of $X$ (but not quite the whole of it): essentially, we can reconstruct the set of vertices and decorated edges, but not the incidence relation between them.
\item We use Theorem \ref{mainA} to show that equivariant elliptic cohomology is not an invariant of the equivariant derived category of a toric variety (see Theorem \ref{mainB} of this Introduction). This confirms the expectation that elliptic cohomology encodes data that is not 1-categorical in nature, and therefore is not captured by the derived category (which is in a sense the universal repository of 1-categorical information of a variety or stack). Note that  both equivariant K-theory and equivariant cohomology are  derived invariants.
\end{itemize}

\subsection{Equivariant elliptic cohomology and toric varieties}
We will fix thoughout the paper an  elliptic curve $E$ over $\mathbb{C}$.

By definition, elliptic cohomology is a complex oriented  cohomology theory whose formal group law is  the completion of $E$ at the identity. In fact, we will only work with the complexification of elliptic cohomology.   Complexification collapses all cohomology theories to  singular cohomology, up to shifts and sums. However this is no longer true if we work equivariantly with respect to a group action. Building on ideas that ultimately go back to Atiyah--Segal's theorem on the completion of genuine equivariant K-theory, Grojnowski  gave a beautiful construction of complexified equivariant elliptic cohomology in \cite{grojnowski1994delocalised}.

Turning on equivariance amounts to decompleting the formal group law. In the case of elliptic cohomology,  this means that the $S^1$-equivariant elliptic cohomology  of an $S^1$-space $X$ should define a  coherent sheaf of $\mathbb{Z}_2$-graded algebras over $
E$.  %, rather than just algebras as ordinary cohomology theory. 
   More generally, if $G$ is a compact Lie group, $G$-equivariant elliptic cohomology should take values in coherent sheaves of algebras over the moduli of $G_\mathbb{C}$-bundles over $E$, $Bun_{G_{\mathbb{C}}}(E)$. This feature is built in in Grojnowski's  definition of equivariant elliptic cohomology, which we   will   review  in Section \ref{preleqell}.  Note that this story is parallel to the classical picture of complex   K-theory, whose formal group law is the completion of the multiplicative group $\mathbb{G}_m$ at one: indeed, the $S^1$-equivariant K-theory of a space $X$ is a module over $K_{S^1}^0(pt)=\mathbb{Z}[t, t^{-1}]$ and therefore can be viewed as defining a coherent sheaf over $$\mathbb{G}_m=Spec(\mathbb{Z}[t, t^{-1}])$$% Atiyah-Segal's work can be interpreted as saying un particular that genuine circle equivariant K-theory  rise to coherent sheaves of algebras over $\mathbb{G}_m$. 
 
Let $T$ be a complex algebraic torus of rank $n$, and let $T^\vee$ be its cocharacter lattice. We define 
$\, 
\cM_T := E\otimes T^\vee.$  
Choosing a basis of $T$ gives an  isomorphism $\cM_T \cong E^n$. Let $X$ be a $T$-toric variety. The equivariant elliptic cohomology of $X$ is a coherent sheaf\footnote{Strictly speaking the equivariant elliptic cohomology of a $T$-space is a  $\mathbb{Z}_2$-graded complexes of coherent sheaves over $\cM_T$. However in the case of GKM varieties with vanishing odd cohomology, of which smooth toric varieties are an example, all information is encoded in the degree zero part which is just an ordinary coherent sheaf; see Remark \ref{truncation} for more comments on this point.}
$$
\Ell_T(X) \in Coh(\cM_T)$$ 
 %We will assume  that $X$ is \emph{good} in the sense of Definition \ref{goodtoric} in the main text: that is, $X$  has a toric affine open cover by affine spaces. For instance, all smooth proper toric varieties are good. % We denote by $\Sigma_X$ the fan of $X$. 
Our first main result shows that $\Ell_T(X)$, viewed just as a coherent sheaf over $\cM_T$, encodes interesting information on the geometry of $X$ and  the $T$-action. 

\begin{customthm}{A}[Theorem \ref{main1}]
\label{mainA}
Let $X$ be a smooth and proper toric variety. Then    $\Ell_T(X)$ is a vector bundle of rank $m$ where $m$ is the number of torus fixed points.  %, that is the number of top-dimensional cones in $\Sigma_X$. 
Additionally, 
we can reconstruct  from $\Ell_T(X)$
\begin{enumerate}
\item the number of one-dimensional orbits $O$ of the $T$-action
\item and, for every one-dimensional orbit $O$, the $n-1$-dimensional isotropy torus $T' \subset T$
\end{enumerate} 
\end{customthm}
In fact Theorem   \ref{main1} in the main text is  more general than Theorem \ref{mainA}, as it applies to the wider class of toric varieties that are \emph{good} in the sense of Definition \ref{goodtoric}. Good toric varieties have a toric open cover by affine spaces. The statement has to  be slightly modified, as only the relatively compact one-dimensional orbits are visible through $\Ell_T(X)$. This is a minor generalization, but it gives us the possibility to apply our result to non-proper toric varieties, that typically have a simpler  geometry. In particular Example \ref{examplenoniso}, which underlies Theorem \ref{mainB}, is about a pair of good  toric three-folds that are not proper.

%Before proceedings, in the following remarks, we comment on the relationship between our theorem and previous results; and on the differences   between equivariant elliptic cohomology  and equivariant $K$-theory and singular cohomology.  

\begin{remark}
In Theorem \ref{mainA} we regard $\Ell_T(X)$ purely as a coherent sheaf over $\cM_T$, and disregard the algebra structure. In fact, keeping track of the algebra structure of  $\Ell_T(X)$ would  allows us  to reconstruct $X$ and the $T$-action entirely. Masuda has proved in \cite{masuda2008equivariant} that ordinary equivariant cohomology $H^*_T(X)$ with its algebra structure is a complete invariant of a toric variety $X$. It is easy to see that the same is true of $\Ell_T(X)$, which is a richer invariant than $H^*_T(X)$, and recovers the latter as its completion at the identity element of the abelian variety $\cM_T$. 
\end{remark}

\begin{remark}
Equivariant K-theory and equivariant cohomology  also define coherent sheaves on the decompletions of their formal group laws (see Section \ref{preleqell} for more details): 
$$\mathcal{K}_T(X) \in Coh(\mathbb{G}_m \otimes T^\vee)  \quad \text{and}  \quad \mathcal{H}_T(X) \in Coh( \mathbb{A}^1 \otimes T^\vee)$$
 It is easy to see that, just as $\Ell_T(X)$,  $\mathcal{K}_T(X)$ and $\mathcal{H}_T(X)$ are vector bundles  of rank $m$ where $m$ is the number of torus fixed points. However, contrary to $\Ell_T(X)$, they are necessarily trivial vector bundles. Thus the number of torus fixed points is the only invariant of the $T$-action that can read off  $\mathcal{K}_T(X)$ and $\mathcal{H}_T(X)$.
\end{remark}

Theorem \ref{mainA} is probably the best result that one can hope for for (good) toric varieties. For instance, example \ref{toricsurf} implies  that the full GKM graph of $X$ cannot be reconstructed from $\Ell_T(X)$. It is an interesting question whether Theorem \ref{mainA}  generalizes to all GKM-manifolds. We remark that our proof depends on the simple combinatorics of fans, which is not available in the general GKM setting. On the other hand the number of fixed points can always be read off $\Ell_T(X)$ for all GKM manifolds, since it is equal to the rank of $\Ell_T(X)$. \begin{question}
Let $X$ be a GKM $T$-manifold. Can we 
reconstruct  the number of one dimensional orbits, and the corresponding  isotropy groups, from $\Ell_T(X)$? 
\end{question}  
Before proceeding let us  comment on the proof of Theorem \ref{mainA}. Let 
$X$ be a good toric variety, and let 
$\mathfrak{U}$ be its toric open cover by affine spaces. The key point is showing that, given $\Ell_T(X)$, we can reconstruct uniquely  the 2-truncation of the \v{C}ech complex that computes $\Ell_T(X)$  as a limit of $\Ell_T(U)$, where $U$ ranges between the finite intersections of  open subsets in $\mathfrak{U}$. This immediately  gives us access to the desired  information on the one-dimensional orbits of the $T$-action. However proving this fact requires a rather laborious calculation of the coherent cohomology of $\Ell_T(X)$. Namely our argument requires showing that the  dimension of the top coherent cohomology of $\Ell_T(X)$ is equal to the number of torus-fixed points.  Our methods are essentially elementary. The calculation is carried out   in Section \ref{sccc} and depends on finding a small, combinatorially manageable, model of the \v{C}ech complex that calculates the coherent cohomology of $\Ell_T(X)$. 

\subsection{Equivariant elliptic cohomology and derived equivalences}
Let $X$ be a scheme. The derived category of coherent sheaves $\mathrm{D^bCoh}(X)$ encodes a  great deal of information on the geometry of $X$. This is the rationale behind a much studied approach to non-commutative geometry, advocated by Kontsevich and others, whereby  general triangulated categories\footnote{The correct   formalism  is actually provided by triangulated dg categories, or stable $\infty$-categories, rather than by classical triangulated categories. In this introduction, for the sake of clarity, we blur this distinction; the reader can find more discussion of this point in Section \ref{prelimcat}.}  are viewed as non-commutative spaces,  and one develops techniques   to extract geometric information directly from them. There is a vast literature on this subject; we limit ourselves to mention  in this connection the beautiful theory of non-commutative motives studied in   \cite{blumberg2013universal} and   \cite{robalo2015k}. %Many cohomological invariants of $X$ can be read off $\mathrm{D^bCoh}(X)$. 
 %This is formalized by the beautiful theory of non-commutative motives, whereas in \cite{} attempts are made to reconstruct from the derived category also the Hodge filtration. 

If we work over the complex numbers,  both the $\mathbb{Z}_2$-periodized  ordinary cohomology and the topological K-theory of a smooth variety $X$ can be computed from $\mathrm{D^bCoh}(X)$. Let $X^{an}$ be the analytification of $X$. It is classical that the periodic cohomology of $X$ is equivalent to the Tate fixed points of the natural $S^1$-action on the Hochschild chains of $\mathrm{D^bCoh}(X)$
$$
HP^*(X^{an}) \simeq HH^*(\mathrm{D^bCoh}(X))^{Tate}
$$
As for K-theory, work of Blanc \cite{blanc2016topological} shows that 
$$
K(X^{an}) \simeq {\bf K}^{top}(\mathrm{D^bCoh}(X))
$$
where ${\bf K}^{top}$ is a  \emph{localizing  invariant} of  triangulated categories in the sense of \cite{blumberg2013universal}. These results carry over to the equivariant setting, which is most relevant for our purposes. Namely, it was showed in \cite{halpern2020equivariant} that, if 
$G$ is a reductive algebraic group acting on 
$X$, the complexified  $G$-equivariant topological K-theory of $X$  can be computed from the equivariant derived category of $X$; note that, essentially by definition, the latter 
is equivalent to the derived category of the quotient stack
$$
\mathrm{D^bCoh}_G(X)\simeq \mathrm{D^bCoh}([X/G])
$$
As a consequence,  $G$-varieties that are $G$-equivariantly derived equivalent have isomorphic equivariant K-theory. This easily implies that also their Borel equivariant periodic cohomology are isomorphic.

In contrast, heuristics coming from chromatic homotopy theory predict that elliptic cohomology should not be an invariant of the derived category.  
One of our main motivations for  proving  Theorem \ref{mainA} was 
to develop a computational tool which would allow us to verify this prediction.  This is our second main result.

\begin{customthm}{B}[Example \ref{examplenoniso}]
\label{mainB}
There exist pairs of toric varieties $X$ and $X'$ such that $X$ and $X'$ are $T$-equivariantly derived equivalent\footnote{An equivariant derived equivalence is \emph{not} the same as a derived equivalence between $\mathrm{D^bCoh}([X/T])$ and  $\mathrm{D^bCoh}([X'/T])$: additionally, one needs to require compatibility with the natural $\mathrm{D^bCoh}([pt/T])$-action. We refer the reader to Section \ref{derivedequivalences} for more precise statements.} 
$$
\mathrm{D^bCoh}_T(X) \simeq \mathrm{D^bCoh}_T(X')  
$$
but they have non-isomorphic elliptic cohomology. As a consequence,  equivariant elliptic cohomology is not a derived invariant.
\end{customthm}
Given Theorem \ref{mainA} it is easy to see that, in fact, such examples are plentiful and easily constructed. Atiyah  flops provide a wealth of examples of $T$-equivariantly derived equivalent toric varieties in all dimensions. For three-folds this amounts to flipping an arc in  (a 2d projection of the) fan, and this has the effect of  %substituting one of the codimension-one isotropy subtori of the $T$-action with a different codimension-one subtorus. 
%making one substitution in
 taking out one of the isotropy subtori of the one-dimensional orbits, and replacing it with a different codimension-one subtorus.   % and replacing it with a different one.  
By Theorem \ref{mainA}, the collection of codimension one isotropy subtori  can be read off the equivariant elliptic cohomology of a toric variety.  This implies that  equivariant elliptic cohomology is  not a derived invariant. 

 {\bf Acknowledgements:}
 This article originates from a seminar on elliptic cohomology at SISSA that was organized by the first author   in 2020. We thank Barbara Fantechi, Hayato Morimura, Andrea Ricolfi and Paolo Tomasini for their   enthusiastic participation in the seminar, and their willingness to  learn about this wonderful subject together with us. We also thank Nora Ganter, Margherita Lelli--Chiesa and James Pascaleff for useful  discussions on the topic of this paper. 
  \section{Preliminaries}
 \label{preliminaries}
 Throughout the paper we will work  over the field of complex numbers $\bC$.
\subsection{Equivariant elliptic cohomology}
\label{preleqell}
We will study the equivariant elliptic cohomology of toric varieties with respect to the action of a maximal algebraic torus $T$. We give an overview of the construction of complexified  equivariant elliptic cohomology following Ganter's account in \cite{ganter2014elliptic}, to which we refer the reader for motivations and additional details. We restrict attention to torus actions, as this simplifies somewhat the exposition. The construction was first proposed in Grojnowski's seminal paper \cite{grojnowski1994delocalised}, with later contributions by Rosu, Ando and others \cite{rosu2003equivariant}, \cite{ando2003sigma}.   

We fix a complex elliptic curve $E$. Let $T$ be a finite dimensional real torus, and let $X$ be a finite $T$-CW complexes. In this paper $X$ will always be a complex toric variety equipped with the action of a maximal torus, and therefore $X$ will be in particular (equivariantly homotopic to) a finite $T$-CW complex.   Let 
$$
M_T = Hom(T, U(1)) \quad \text{and} \quad N_T = Hom(U(1),T)
$$
the character and the cocharacter lattice of $T$. Consider the variety
$$
\mathcal{M}_T = E \otimes N_T
$$
Note that any isomorphism $N_T \cong \mathbb{Z}^n
$, where $n$ is the rank of $T$, induces an isomorphism $\cM_T  \cong E^n$. Following Ganter \cite{ganter2014elliptic}, we will define the equivariant elliptic cohomology of $X$ to be  a coherent sheaf of commutative algebras over $\mathcal{M}_T$ 
$$
\cE ll_T(X) \in Coh(\mathcal{M}_T)
$$

\begin{remark}
\label{truncation}
Elliptic cohomology is an even periodic cohomology theory, see  e.g. \cite{lurie2009survey}. We 
should expect the equivariant elliptic cochains of  a space 
$X$ to define a $\mathbb{Z}_2$-periodic complex of coherent sheaves over $\mathcal{M}_T$, rather than just a coherent sheaf.   
This is indeed the case:  
 in general, the coherent sheaf $\cE ll_T(X)$  described   below only computes the $0$-th cohomology of the full equivariant elliptic cochain complex of $X$. However for toric varieties this entails no loss of information. Toric varieties have vanishing odd equivariant cohomology. As a consequence their full equivariant elliptic cohomology splits as a direct sum
$$
\bigoplus_{k \in \mathbb{Z}} \cE ll_T(X)[2k]
$$
It is therefore sufficient to keep track of the degree 0 summand and, following  the conventions in \cite{ganter2014elliptic}, this is what we will do in the remainder of the paper.
\end{remark}

We denote $\mathcal{H}_T(X)$ the (even) Borel equivariant cohomology of $X$ with complex coefficients, and  grading collapsed  modulo two 
$$
\mathcal{H}_T(X) := \bigoplus_{n \in \mathbb{Z}} H^{2k}(X \times_TET, \mathbb{C})
$$
When $X=\mathrm{pt}$, $
\mathcal{H}_T(X)$ is naturally identified with the symmetric algebra over the complexification of the dual of the Lie algebra of $T$. That is, there are identifications 
$$
Spec(\mathcal{H}_T(pt)) \cong \mathfrak{t}_\mathbb{C}  \cong \mathbb{A}^n_\mathbb{C}
$$
where the latter depends on a choice of 
isomorphism $N_T \cong \mathbb{Z}^n
$. Note that in particular, since $
\mathcal{H}_T(X)$ is a finitely presented module over $
\mathcal{H}_T(pt)$, it defines a coherent sheaf 
over  $\mathfrak{t}_\mathbb{C}  \cong \mathbb{A}^n_\mathbb{C}$. 

%Grojnowski's insight was that the stalks of equivariant elliptic cohomology with complex coefficients 
%should compute the Borel equivariant cohomology of (fixed points) of $X$. This idea ultimately goes back to Atiyah and Segal's work on the relationship between genuine and Borel equivariant K-theory.  
% Locally over  $\mathcal{M}_T$ the structure of $
%\cE ll_T(X)$ is therefore well understood. Grojnowski then goes on to show that these local data can be glued together to obtain a global coherent sheaf $
%\cE ll_T(X)$  on $\mathcal{M}_T$. 
%%Below we describe the stalks of $\cE ll_T(X)$ on  $Coh(\mathcal{M}_T)$, and then briefly explain the globalization step. 
%We remark that Grojnowski's construction of $
%\cE ll_T(X)$ exploits the analytic topology of $E$. 
%Since $E$ is proper, the fact that $
%\cE ll_T(X)$ is actually an algebraic coherent sheaf then follows via GAGA. 

Let us describe the stalks of $\cE ll_T(X)$ on  $\mathcal{M}_T$, and then briefly explain the globalization step required to complete the construction. 
An inclusion of tori $T' \to T$ induces a closed embedding $\mathcal{M}_{T'} \subset \mathcal{M}_{T}$. We denote by $X^{T'}$ the $T'$-fixed points for the induced action of $T'$. For every closed point  $\alpha \in \mathcal{M}_{T}$, we denote by $S(\alpha)$  the set of the sub-tori $T' \subset T$ such that $\alpha$ lies in $\cM_{T'} \subset \cM_T$. We set  
$$
T(\alpha) := \cap_{T' \in S(\alpha)} T'
$$
The stalk of $
\cE ll_T(X)$ at $\alpha$ is given by 
\begin{equation}
\label{stalk}
\cE ll_T(X)_\alpha \cong 
\mathcal{H}_T(X^{T(\alpha)}) \otimes_{\mathcal{O}_{\mathbb{A}^n_\mathbb{C}}} ( \mathcal{O}_{\mathbb{A}^n_\mathbb{C}} )_0
\end{equation}
That is, the stalk of $
\cE ll_T(X)$ at $\alpha$ is naturally identified with the stalk of the Borel equivariant cohomology of $X^{T(\alpha)}$ at $0 \in \mathfrak{t}_{\mathbb{C}} \cong \mathbb{A}^n_{\mathbb{C}}$. This is expected on general grounds from equivariant homotopy theory, and is closely related to Atiyah--Segal's work on the relationship between genuine and Borel equivariant K-theory \cite{atiyah1969equivariant}.

To define a global coherent sheaf on $\mathcal{M}_T$ the information on its stalks is not sufficient. 
The key point is that identification (\ref{stalk}) can be 
spread to a sufficiently small analytic neighborhood $U^\alpha$ of $\alpha$: sections of $
\cE ll_T(X)$ on $U^\alpha$ can be identified with sections of $\mathcal{H}_T(X^{T(\alpha)})$ on an appropriate analytic neighborhood of the origin. The sheaf $
\cE ll_T(X)$ can then be defined uniquely by describing its sections, and restrictions, on an analytic cover of $\cM_T$. Since $\cM_T$ is proper, the fact that $
\cE ll_T(X)$ is actually an algebraic coherent sheaf  follows via GAGA. 
We refer the reader to \cite{ganter2014elliptic} for the details. 

The upshot of the construction is that one can define 
 $\cE ll_T(X)$ as a coherent sheaf of commutative algebras over $\cM_T$. 
 \begin{definition}
 We denote by $p:\cM_X \to \cM_T$ the relative Spec of 
 $\cE ll_T(X)$. 
 \end{definition}
Note that by definition there is an isomorphism 
 $\cE ll_T(X) \cong p_*(\cO_{\cM_X})$. Note also that if $X$ is the point with the trivial action, then $\cM_{pt}=\cM_T$.

 We conclude by stating the standard functoriality property of $\cE ll_T(X)$ with respect to  maps of  tori.  If $f:T' \to T$ is a map of tori, there is an induced map $\rho_f: \cM_{T'} \to \cM_T$.
\begin{proposition}
\label{prop: subtorus}
Let $X$ be a finite $T'$-$CW$-complex. Let $f:T' \to T$ be a map of tori. Then $T$ acts on $X\times_{T'} T$ and there is a natural  isomorphism  $$ 
\cE ll_{T}(X\times_{T'} T) 
\stackrel{\cong} \to 
(\rho_f)_* \cE ll_{T'}(X)$$ 
\end{proposition}
\begin{proof}
This is a general property of equivariant cohomology theories, but we sketch a proof for the convenience of the reader. Consider the  $T'$-equivariant map 
$$
i:X \to X\times_{T'} T, \quad  x \in X \mapsto \overline{(x, 1_{T'})} \in X\times_{T'} T
$$
Pull-back in elliptic cohomology yields a map of coherent sheaves  over $\cM_{T'}$ 
$$
\cE ll_{T'}(X\times_{T'} T) \to \cE ll_{T'}(X)
$$
Composing this with the natural isomorphism
$$
(\rho_f)^*\cE ll_{T}(X\times_{T'} T) \cong \cE ll_{T'}(X\times_{T'} T)
$$
we obtain a map 
$ 
(\rho_f)^*\cE ll_{T}(X\times_{T'} T) \to \cE ll_{T'}(X)$ in $Coh(\cM_{T'})$; and, by adjunction, a map in $Coh(\cM_{T})$, $$ \cE ll_{T}(X\times_{T'} T) \to (\rho_f)_*\cE ll_{T'}(X)$$
We need to prove that this last map is an isomorphism, and this can be checked on stalks. Thus we reduce to the analogous statement in equivariant singular cohomology, which is well known. 
%\textcolor{red}{Let us sketch the calculation. 
%Let $\alpha$ be a closed point of $\cM_T$ which is not in the image of $\rho_f$.  Then $T(\alpha)$ does not contain the subtorus $f(T')$. Thus  $X\times_{T'} T$ has no fixed points under the action of $T(\alpha) \subset T$, and thus
%$$
%\cE ll_{T}(X\times_{T'} T)_{\alpha} = 0 = ((\rho_f)_*\cE ll_{T'}(X))_{\alpha}
%$$
%For $\alpha \in E_{T'}$, we find that $T(\alpha) \subset T'$ and therefore $(X\times_{T'} T)^{ T(\alpha)} \simeq X^{T'(\alpha)} \times_{T'} T$. The equivariant cohomogy satisfies $$H_T^*( (X\times_{T'} T)^{T(\alpha)} ) \simeq H_T^*( X^{T'(\alpha)} \times_{T'} T)\simeq H_{T'}^*( X^{T'(\alpha)} ) \otimes_{H_{T'}^*} H_T^*$$}
\end{proof}

\subsubsection{GKM manifolds}
\label{equivarianttoric} 
From now on we will assume that 
$T$ is a complex, rather than a real, torus.  From the viewpoint of  equivariant homotopy theory, this makes no difference as complex algebraic tori are homotopy equivalent to their unitary subgroups, which are compact real tori.

 A  space $X$ equipped with a $T$-action 
 is \emph{equivariantly formal} if the spectral sequence $$
H^p(BT; H^q(X, \mathbb{R})) \Rightarrow H^{p+q}_T(X, \mathbb{R})
$$
induced by the fibration 
$X \times_T BT \to BT$  
 collapses.   We say that an even dimensional, compact manifold $X$ equipped with a $T$-action is  a \emph{GKM manifold} if
 \begin{itemize}
 \item $X$ is equivariantly formal
 \item The $T$-action has only finitely many  zero- and one-dimensional orbits
 \end{itemize}
Complex smooth toric varieties with the action of the maximal torus are GKM manifolds. By Theorem 1.2.2 of \cite{goresky1998equivariant} if $X$ is a GKM manifold, its equivariant cohomology can be computed from  its \emph{GKM graph} (see Definition \ref{gkmgraph} below) built from the zero- and one-dimensional orbits of the $T$-action.  This theorem was generalized to K-theory and elliptic cohomology by Knutson--Rosu \cite{rosu2003equivariant}. We formulate it in the setting of elliptic cohomology as Proposition \ref{pushout} below.  

Let us assume that $X$ is a GKM manifold. Let $\lambda: T \to \bC^*$ be a character of $T$. We  denote by $S^2_\lambda$ the \emph{representation sphere} for $\lambda$, i.e. the one point 
compactification of 
$\bC$ equipped with the $T$-action given by $\lambda$.   Each 1-dimensional orbit $O$ of the $T$-action on $X$ compactifies to a representation sphere: i.e. $
\overline{O} \cong S_{\lambda}$ for some character $\lambda$, 
where $\overline{O}$ is the closure of $O$ inside $X$. %In particular $\overline{O} - O$ consists of two distinct fixed points of the $T$-action.  
 If $T \stackrel{\lambda} \to  \mathbb{C}^*$  is a character of $T$, we denote $T_\lambda :=ker{\lambda} 
$. 
The datum of $\lambda$ is the same as a that of a homomorphism $\phi_\lambda: N_T \to \mathbb{Z}$. If $N_\lambda$ is the kernel of $\phi_\lambda$, then  $T_\lambda$ is equal to $ker(\phi_\lambda) \otimes \mathbb{C}^*$.

\begin{definition}
\label{gkmgraph}
The \emph{moment graph} $\Gamma$ of $X$ is an edge-labelled graph such that 
\begin{enumerate}
\item as vertices, the fixed points of the $T$-action
\item as edges, the one-dimensional orbits: an edge $\cO$ joins two vertices $v_1$ and $v_2$ if they lie on $\overline{\cO}$
\item edges are labelled by characters of $T$: an edge $\cO$ is labelled by $\lambda$ if $\overline{\cO}\cong S^2_\lambda$
\end{enumerate}
\end{definition}
We denote respectively  by $V_\Gamma$ and $E_\Gamma$ the set of vertices and edges of $\Gamma$. We denote by $\lambda_e$ the character labelling an edge $e \in E_\Gamma$, and we set $T_e:=T_{\lambda_e}$.

  \begin{proposition}
  \label{pushout}
  Let $X$ be a GKM manifold.  
 \begin{enumerate}
 \item  The variety $\cM_X$ is the colimit in the category of schemes of the diagram
 $$
 \big [ \coprod_{v \in \Gamma_X} \cM_T \rightrightarrows \coprod_{e \in E_\Gamma} \cM_{T_{e}} \big ] \to \cM_X
$$
where the two arrows are the natural inclusions. 
\item The coherent sheaf $\cE ll_T(X)$ is the limit in $Coh(\cM_T)$ of the diagram of coherent sheaves of algebras
 $$
 \cE ll_T(X) \to \big [ 
  \bigoplus_{v \in \Gamma_X} \cO_{\cM_T}  
 \rightrightarrows  
 \bigoplus_{e \in V_\Gamma}  \cO_{\cM_{T_e}}
 \big ] 
$$
where the arrows are given by pull-back along the arrows in statement $(1)$. 
\end{enumerate}
 \end{proposition}
 \begin{proof}
 Statement $(2)$ is a special case of Corollary 4.3 of \cite{ganter2014elliptic}, which  is a generalization of Theorem 1.2.2 of \cite{goresky1998equivariant} to equivariant K-theory and equivariant elliptic cohomology. Statement $(1)$ follows from $(2)$ by taking relative Spec. This implies  that $\cM_X$ is the colimit in the category of schemes that are affine over $\cM_T$, to prove that it is also the colimit in schemes requires an extra argument. %The only point that needs to clarified is why the colimit in the category of affine schemes over $\cM_T$ is the actual colimit in the category of schemes. 
 Note that the diagram is an iterated push-out of closed immersions. So on every affine patch of $\cM_T$ we can apply Theorem 3.4 of \cite{schwede2005gluing} which states that, under these assumptions, the colimit in the category of affine schemes coincides with the colimit in the category of schemes. It is easy to see that Schwede's result extends to our relative setting by taking the realization of the \v{C}ech nerve of an affine open cover of $\cM_T$, which is  a colimit in the category of schemes. 
  \end{proof}

\begin{remark}
As we mentioned Corollary 4.3 of  \cite{ganter2014elliptic} applies also to equivariant cohomology and equivariant K-theory. Let us explain the case of K-theory. The (degree 0 part of the) equivariant K-theory of $X$, $\cK_T(X)$, defines a coherent sheaf of algebras over 
$$
\cT:=Spec(\cK_T(pt)) \cong T
$$
Set $\cT_\tau := \mathbb{C}^* \otimes \langle \tau \rangle$, and $\cT_X:=Spec(\cK_T(X))$. Then statements $(1)$ and $(2)$ of Proposition \ref{pushout} remain valid for the equivariant K-theory $\cK_T(X)$, if we replace $\cM_X$, $\cM_\sigma$ and $\cM_\tau$ with 
$\cT_X$, $\cT_\sigma$ and $\cT_\tau$.
\end{remark} 
 
 \subsubsection{Categories}
 \label{prelimcat}
 Equivariant elliptic cohomology defines an object in $Coh(\cM_T)$, which is a small abelian category. To carry out some parts of our argument, however, it will be convenient to work with  \emph{triangulated} categories, such as the derived category of quasi-coherent sheaves. In fact, the formalism of triangulated categories is not  quite sufficient for our purposes. We will work instead inside the  $\infty$-category of \emph{triangulated} and \emph{presentable} $\mathbb{C}$-linear dg categories. Triangulated dg categories are by now a familiar replacement of triangulated categories, with better formal properties; viewing them as objects in an $\infty$-category allows to make sense of \emph{limits} and \emph{colimits} of triangulated categories. We stress that these techniques will play a rather small role in this paper, being limited to the proof Lemma \ref{h*lim} in Section \ref{sccc}; and to the formulation of some results in Section \ref{derivedequivalences}. % where essentially the more classical formalism of triangulayted we express in terms of an equivalence of dg categories the notion of derived equivalence of stacks. 
  The subtleties of $\infty$-categories will play no actual role in our argument. Here we limit ourselves to introduce notations, and state for future reference a descent property. A standard reference for this  material is \cite{gaitsgory2019study}.
 \begin{enumerate}
 \item Let $\mathrm{Pr}^{\mathrm{L, tr}}$ the $\infty$-category of presentable and triangulated dg categories
 \item If $R$ is a dg algebra, we denote by $R\text{-}\mathrm{mod}$ the presentable triangulated dg category of $R$-modules. In our applications, $R$ will always be equal to $\bC$
 \item If $X$ is a stack, we denote by 
 $\Qcoh(X)$ the presentable triangulated dg category of quasi-coherent sheaves over $X$. This is a dg enhancement of the classical unbounded derived category of quasi-coherent sheaves over $X$. In our applications, $X$ will always be either a smooth variety, a smooth DM stack or a smooth Artin stack
 \item We denote by $\Perf(X)$ the full-subcategory of compact objects in $\Qcoh(X)$. This is a dg enhancement of the classical triangulated category of perfect complexes over $X$. Both $\Qcoh(X)$ and $\Perf(X)$ carry a symmetric monoidal structure
 \item If $f: X \to Y$ is a map of stacks, we have a pair of adjoint functors 
 $$
 f^*: \Qcoh(Y) \leftrightarrows \Qcoh(X): f_* 
 $$
 \item The functor $f^*$ is symmetric monoidal, and restricts to the subcategory of perfect complexes. This it induces a monoidal action of $\Qcoh(Y)$ on $\Qcoh(X)$, and of $\Perf(Y)$ on $\Perf(X)$
 \item Now assume that $Y$ is smooth, and that $f$ is proper. Then both functors $f^*$ and $f_*$ restrict to the subcategories of perfect complexes
 $$
 f^*: \Perf(Y) \leftrightarrows \Perf(X): f_* 
 $$
 \end{enumerate}
\begin{proposition}
\label{descent}
The assigment sending a scheme $X$ to the category $\Qcoh(X) \in \mathrm{Pr}^{\mathrm{L, tr}}$, and a map $f: X \to Y$ to the pull-back $f^*:\Qcoh(Y) \to \Qcoh(X)$, can be promoted to a colimit preserving contravariant functor. In particular if $X$ is isomorphic to a  colimit in the category of schemes 
$
X \cong \varinjlim_{i \in I} U_i
$
with structure maps $\alpha_{ij}:U_i \to U_j$, then $\Qcoh(X)$ is equivalent to a limit in $\mathrm{Pr}^{\mathrm{L, tr}}$ 
$$
\Qcoh(X) \simeq \varprojlim_{i \in I} \Qcoh(U_i)
$$
with structure maps $(\alpha_{ij})^*:\Qcoh(U_j) \to \Qcoh(U_i)$.
\end{proposition}
\section{Equivariant elliptic cohomology of toric varieties}
Throughout this section we fix a complex elliptic curve $E$. We let $T$ be a $n$-dimensional algebraic torus, and consider toric varieties with respect to a $T$-action.   We use the notations introduced in Section \ref{preliminaries}.

We start by briefly recalling basic notions from toric geometry. For a comprehensive introduction to toric geometry, we refer the reader to \cite{fulton1993introduction}. The main result of this Section is Theorem \ref{main1}, which shows that equivariant elliptic cohomology allows us to recover highly non-trivial  information on fixed points and 1-dimensional orbits of the $T$-action (although not the full GKM graph).

We denote by $M_T$ and $N_T$,  respectively, the character and cocharacter lattice of  $T$. 
Let $n$ be the dimension of $T$. Let $\Sigma \subset N_T$  be a fan. That is, $\Sigma$ is a set of strongly convex cones closed under taking faces and intersections. We can associate to a fan $\Sigma$ a 
$n$-dimensional complex  toric variety $X_\Sigma$, which is equipped with an action of $T$. Conversely, if $X$ is a toric $T$-variety,  we can associate to $X$  a fan $\Sigma_X$ in $N_T$.  %will denote  by $\Sigma_X$ the fan of a toric variety $X$. 

If $\tau$ is a cone in $\Sigma$, we denote by $\langle \tau \rangle \subset N_T$ the sublattice generated by the vectors in $\tau$. For all   $\tau \in \Sigma$ we set $U_{\tau}:= \Spec\  \bC[\tau^\vee]$, where $\tau^\vee:= \{c \in M_T \mid \tau(c) \ge 0\}$ is the dual cone.  For all $0 \leq k \leq n$, 
we denote by $\Sigma^k$ the collection of $k$-dimensional cones in the fan $\Sigma$. Occasionally, we will denote  the set of top-dimensional cones of $\Sigma$ also by $\Sigma^{top}$. Cones of $\Sigma$ are in bijection with torus-equivariant affine open subsets of $X_\Sigma$: if $\tau$ is a cone, we denote by $U_\tau \subset X_\Sigma$ the corresponding open subset.

%The toric variety $X=X_\Sigma$ associated to the fan $\Sigma$ is isomorphic to the following colimit of schemes along open inclusions 
%$$ \bigsqcup_{\sigma, \sigma' \in \Sigma^{\mathrm{top}} }U_{\sigma \cap \sigma'} \rightrightarrows  \bigsqcup_{\sigma \in \Sigma^{\mathrm{top}}} U_{\sigma}$$ 
%
%

% If $X=X_\Sigma$ we sometimes denote $N_k$ as $N_{X,k}$ to emphasise the dependence on $X$.
  
Throughout the paper we make the following two assumptions on the toric variety 
$X=X_\Sigma$
\begin{enumerate}
\item $X$ is smooth. Equivalently, for all $\sigma \in \Sigma^{top}_X$ the sublattice $\langle \sigma \rangle$ is equal to  $N_T$
\item $X$ has a toric affine open cover by  open subsets  which are isomorphic to $\mathbb{A}^n$. Equivalently, all cones of $\Sigma_X$ lie on a top-dimensional cone
\end{enumerate}
\begin{definition}
\label{goodtoric}
We say that a toric variety $X$ is \emph{good} if it satisfies conditions $(1)$ and $(2)$.
\end{definition}

If $\tau$ is a cone of $\Sigma$, we set 
$\cM_\tau:= E \otimes \langle \tau \rangle$. Note the following special cases
\begin{itemize}
\item If $\tau$ is maximal, then $\cM_\tau=\cM_T$
\item If $\tau=0$, then $\cM_\tau = pt$
\end{itemize}
An inclusion of cones  
$\tau' \subset \tau$ induces a closed embedding $\cM_{\tau'} \subset \cM_\tau$. 

\begin{lemma}
\label{lemma:elliptic cones} 
Let $\tau$ be  a cone in $\Sigma_X$. We view $\cM_\tau$ as a closed subvariety of $\cM_T$ via the embedding  induced by the inclusion 
$\langle \tau \rangle \subset N_T$. 
Then there is an isomorphism $$\cE ll_T(U_{\tau}) \cong \cO_{\cM_{\tau}}$$  
\end{lemma} 
\begin{proof}
Let $m$ be the rank of $N_\tau$. There is an  embedding  $( \bC^{*})^m \cong \langle \tau \rangle \otimes \mathbb{C}^* \to T = N_T \otimes \mathbb{C}^*$. We have isomorphisms $$U_\tau  \cong \bC^{m} \times_{ ( \bC^{*})^m} T \cong   \bC^{m} \times  ( \bC^{*})^{n-m}
$$ 
Thus it follows from Proposition \ref{prop: subtorus} that $$\cE ll_T(U_{\tau}) \cong i_*\cE ll_{(\bC^*)^m} (\bC^m) \cong     \cO_{\cM_\tau}$$
where the first isomorphism comes from the fact that  $\bC^m$ is equivariantly homotopic to the point.
\end{proof}

We will give an elementary proof of Proposition  \ref{pushout} for toric varieties, based on the properties of the torus equivariant affine cover. See Lemma \ref{mayervietoris} and Corollary \ref{pushouttoric} below. This will allow us  in particular to extend  Proposition  \ref{pushout} to good toric varieties which are not necessarily compact, and therefore are not strictly speaking GKM manifolds: whereas compactness in assumed in Proposition  \ref{pushout}.

Let $(\Sigma^{n-1}_X)^* \subset 
\Sigma^{n-1}_X$ be the subset of $n-1$-dimensional cones $\tau$ that lie on two different top dimensional cones $\sigma$ and $\sigma'$. Note  that if 
$X$ is proper, i.e. $\Sigma_X$ is a complete fan, $(\Sigma^{n-1}_X)^*$ is equal to 
$\Sigma^{n-1}_X$.

\begin{lemma}
\label{mayervietoris}
Let $X$ be an $n$-dimensional toric  variety. We enumerate from 1 to $m$ the maximal cones of 
$\Sigma_X$
$$
\Sigma_X^{top}=\{\sigma_i\}_{i \in I}, \quad 
I = \{1, \ldots, m\}
$$
%\begin{enumerate}
%\item 
There is an exact sequences in $Coh(\cM_T)$ 
\begin{equation}
\label{longexact}
0 \to \cE ll_T(X) \stackrel{\rho_1}\to \bigoplus_{i \in I} \cO_{\cM_T} \stackrel{\rho_2}\to 
\bigoplus_{i, j \in I, i < j} \cO_{\cM_{\sigma_i \cap \sigma_j}} \to \ldots   \to 
 \cO_{\cM_{\sigma_1 \cap \ldots \cap \sigma_m}}  \to 0
\end{equation}
Further, away from a closed subscheme of codimension $2$,  $\rho_2$ induces an isomorphism 
$$
coker(\rho)  \cong  \bigoplus_{\tau \in (\Sigma^{n-1}_X)^*} \cO_{\cM_\tau}
$$
%$$
%\bigcup_{\tau \in (\Sigma_X^{n-1})^*}  \cM_{\tau}  \subset \cM_T
%$$ 
%\item $\cE ll_T(X)$ is equivalent to the following limit in $Coh(\cM_T)$ 
%$$
%\cE ll_T(X) \cong \varprojlim_{\sigma \in \Sigma_X} \cO_{\cM_{\sigma}}
%$$
%\end{enumerate}
\end{lemma}
\begin{proof}
Consider the equivariant open cover of $X$ given by 
$\{U_{\sigma_i}\}_{i \in I}$. By Lemma \ref{lemma:elliptic cones}, the elliptic cohomology of $U_{\sigma_1 \cap \ldots \cap \sigma_n}$ is given by $$\Ell_T( U_{\sigma_1 \cap \ldots \cap \sigma_n}) \cong \cO_{\cM_{\sigma_1 \cap \ldots \cap \sigma_n}}$$ 
The first statement is a consequence of Mayer-Vietoris and the vanishing of the odd equivariant elliptic cohomology of smooth toric varieties.  In the base case  $m=2$ we obtain an exact sequence 
$$0 \to \Ell_T(X) \to \Ell_T(U_{\sigma_1}) \oplus \Ell_T(U_{\sigma_2})\to \Ell_T(U_{\sigma_1 \cap \sigma_2}) \to 0.$$ Induction on $m$ yields the sequence 
$$
0 \to \Ell_T(X) \to \bigoplus_{i \in I} \Ell_T(U_{\sigma_i})  \to 
\bigoplus_{i < j} \Ell_T(U_{\sigma_i\cap \sigma_j}) \to \ldots   
\to \Ell_T(U_{\sigma_1 \cap \ldots \cap \sigma_m}) \to 0
$$
We conclude applying again Lemma \ref{lemma:elliptic cones}.

Now let us consider the last  statement. We set 
$$
Z:= \bigcup_{\tau \in \Sigma_X^{n-2}} \cM_\tau \quad  \text{and} \quad U:= \cM_T-Z
$$
Note that $Z$ has codimension  
$2$. After restricting to $U$ we note that
\begin{enumerate}
\item All terms of the exact sequence (\ref{longexact}) vanish except the first three, as they all have support in $Z$
\item For the same reason, the natural inclusion $$ \bigoplus_{i, j \in I, i < j} \cO_{\cM_{\sigma_i \cap \sigma_j}} \to \bigoplus_{\tau \in (\Sigma^{n-1}_X)^*} \cO_{\cM_\tau}$$ becomes an isomorphism
\end{enumerate}  
As restriction to $U$ is an exact functor, this concludes the proof.
%The second  statement follows from the fact that $X$ can be written as the colimit of its torus-invariant open subschemes 
%  $$
%T=U_{0}, \quad 
%\bigsqcup_{\sigma \in N_1} U_\sigma, \quad \ldots\bigsqcup_{\sigma \in N_{n-1}} U_{\sigma}, \quad  \bigsqcup_{\sigma \in N_n} U_{\sigma}   
%$$
%along the natural inclusions. When passing to the analytification this yields a (homotopy) colimit of spaces. When applying elliptic cohomology, this yields a limit in the derived category of  coherent sheaves on $E_T$.  
\end{proof}

\begin{remark}
%As  we will see the properties of 
%$\Ell_T(X)$ as a coherent sheaf over $\cM_T$ encode deep geometric information on the $X$ and the $T$-action. 
We will spend much of the remainder of the paper studying properties of $\Ell_T(X)$ as a coherent sheaf over $\cM_T$. Thus, it might be interesting to remark that Lemma \ref{mayervietoris} allows us to easily compute the \emph{determinant} of 
$\Ell_T(X)$. Namely, there is an isomorphism  
$$
\mathrm{det}(\Ell_T(X)) \cong 
\cO_{E_T}(- \sum_{\tau \in (\Sigma_X^{n-1})^*} \cM_{\tau})
$$ 
 Indeed, by standard properties of determinants, the line bundle $\mathrm{det}(\Ell_T(X))$ can be computed by tensoring together the determinants of all terms, but the first, of sequence (\ref{longexact}). Further $\mathrm{det}(\cO_Z)$ vanishes as soon as $Z$ has codimension two or higher. This implies that only the second term of sequence (\ref{longexact}) contributes to the determinat, and we find that 
 $$
 \mathrm{det}(\Ell_T(X)) \cong \mathrm{det}( 
\bigoplus_{\tau \in (\Sigma^{n-1}_X)^*} \cO_{\cM_\tau}) \cong \bigotimes_{\tau \in (\Sigma^{n-1}_X)^*} \mathrm{det}(\cO_{\cM_\tau})  \cong \cO(- \sum_{\tau \in (\Sigma_X^{n-1})^*} \cM_{\tau})
 $$
\end{remark}

   \begin{corollary}
  \label{pushouttoric}
  Let $X$ be a good toric variety.  
 \begin{enumerate}
 \item  The variety $\cM_X$ is the colimit in the category of schemes of the diagram
 $$
 \big [ \coprod_{\tau \in (\Sigma^{n-1}_X)^*} \cM_{\tau} \rightrightarrows \coprod_{\sigma \in \Sigma_X^{top}} \cM_\sigma \big ] \to \cM_X
$$
where the two arrows are induced by the inclusion of  $\tau \in (\Sigma^{n-1}_X)^*$ in the top dimensional cones containing it. 
\item The coherent sheaf $\cE ll_T(X)$ is the limit in $Coh(\cM_T)$ of the diagram of coherent sheaves of algebras
$$
 \cE ll_T(X) \to \big [ 
  \bigoplus_{\sigma \in \Sigma_X^{top}} \cO_\cM  
 \rightrightarrows  
 \bigoplus_{\tau \in (\Sigma^{n-1}_X)^*}  \cO_{\cM_{\tau}}
 \big ] 
$$
where the arrows are given by pull-back along the arrows in statement $(1)$. 
\end{enumerate}
 \end{corollary}
 \begin{proof}
 As in the proof of Proposition \ref{pushout}, Statement $(1)$ follows from $(2)$ by taking relative Spec. In the case of good toric varieties statement $(2)$  follows from Lemma \ref{mayervietoris}. Indeed, since $Coh(\cM_T)$ is an abelian category, the equalizer can be computed as the kernel of the difference of the two maps, which is the natural morphism
 $$
  \bigoplus_{\sigma \in \Sigma_X^{top}} \cO_\cM  
 \to   
 \bigoplus_{\tau \in (\Sigma^{n-1}_X)^*}  \cO_{\cM_{\tau}}
 $$
It follows from Lemma \ref{mayervietoris} that the kernel of this map is isomorphic to $\Ell_T(X)$. 
  \end{proof}

% \begin{remark}
%Good toric varieties,  in the sense of Definition \ref{goodtoric}, are not proper in general: for examples, affine space is good. Note that Corollary 4.3 of  \cite{ganter2014elliptic} is  formulated for GKM manifolds and these, by definition, are compact. In the compact case $(\Sigma^{n-1}_X)^* = \Sigma^{n-1}_X$, and indeed Ganter's statement is formulated in terms of $\Sigma^{n-1}_X$.
%
%It is easy to see, however, that the statement also applies to good toric varieties: in this setting we still have the standard toric cover by affine spaces, which are equivariantly homotopic to the point, and the gluing pattern is determined by one dimensional orbits. The only difference is that we can discard the information coming from  one-dimensional orbits that are not incident to two distinct fixed points: i.e. we can work with the subset $(\Sigma^{n-1}_X)^* \subset \Sigma^{n-1}_X$, as we do   in  Proposition \ref{pushout} above. \end{remark}

\begin{proposition}
\label{vector bundle}
Let $X$ be a good toric variety. Then $\Ell_T(X)$ is a vector bundle of rank $m:= |\Sigma^{top}_X|$ over $\cM_T$.
\end{proposition}
\begin{proof}
It follows from Corollary \ref{pushouttoric} that the map 
$p:\cM_X \to \cM_T$ is a finite and flat morphism of degree $m$. Under these assumptions, the pushforward of a line bundle is a rank $m$ vector bundle, and this applies in particular to $
p_*(\cO_{\cM_X}) \cong \Ell_T(X). 
$  
\end{proof}

 \begin{lemma}
 \label{main}
Let $X$ be a good toric variety. Then 
\begin{enumerate}
\item The dimension of the space of cosections of $\Ell_T(X)$ is equal to the number of top dimensional cones in $\Sigma_X$: in symbols 
$$\mathrm{dim}_{\bC}\big ( \Hom_{Coh(\cM_T)}( \Ell_T(X), \cO_{\cM_T}) \big ) = |\Sigma^{top}_X|=:m$$
\item Consider the canonical map 
$$
\psi: 
 \Ell_T(X) \to  \cO_{\cM_T} \otimes \big (\Hom_{Coh(\cM_T)}(\Ell_T(X), \cO_{\cM_T})\big )^*
$$ Then, away from a closed subscheme of codimension $2$,  there is an isomorphism 
$$
coker(\psi)   \cong  \bigoplus_{\tau \in (\Sigma^{n-1}_X)^*} \cO_{\cM_\tau}
$$
\end{enumerate}
\end{lemma}
\begin{proof}
Recall that $\Ell_T(X)$ is isomorphic $p_* \cO_{\cM_X}$, where $p: \cM_X \to \cM_T$ is the structure map. By Serre duality, we have that 
$$
\Hom_{Coh(\cM_T)}( p_*\cO_{\cM_X}, \cO_{\cM_T }) \cong (\mathrm{Ext}^n(\cO_{\cM_T }, p_*\cO_{\cM_X}))^* = 
(H^n(p_*\cO_{\cM_X}))^*
$$ Now the map $p$ is affine, and 
therefore the derived pushforward of $p$ coincides with the classical pushforward $p_*$. As a consequence we have an isomorphism 
$$H^n( p_*\cO_{\cM_X})\cong H^n ( \cO_{\cM_X})$$ 
Now the first part of the Lemma follows from Theorem \ref{coho}, that will be proved in Section \ref{sccc}, and states that 
$$
\mathrm{dim}_{\mathbb{C}}(H^n ( \cO_{\cM_X}))= |\Sigma^{top}_X|
$$

Let us prove the second part of the Lemma. There is a canonical map
$$
\psi: \Ell_T(X) \to  \cO_{\cM_T} \otimes \big (\Hom_{Coh(\cM_T)}(\Ell_T(X), \cO_{\cM_T})\big )^*
$$
obtained by dualizing the evaluation map.  
Upon choosing a basis of the vector space   of cosections $\Hom_{Coh(\cM_T)}(\Ell_T(X), \cO_{\cM_T})$, we obtain a map 
$$
\psi': \Ell_T(X) \to 
\bigoplus_{i=1}^{i=m} \cO_{\cM_T} 
$$
Up to the $GL(m, \mathbb{C})$-action on the second factor, $\psi'$ is the unique map between its source and its target to be generically (i.e. away from a closed locus) an isomorphism. Recall indeed that the rank  of $
 \Ell_T(X)$ is also equal to $m$. 
 
 Now, consider the first map in   exact sequence (\ref{longexact}) from Lemma 
 \ref{mayervietoris} 
 $$
 \rho_1: \Ell_T(X) \to 
\bigoplus_{i=1}^{i=m} \cO_{\cM_T} 
 $$
 Note that $\rho_1$ is also generically an isomorphism, thus there is a matrix $A \in GL(m, \bC)$ such that the diagram below commutes
 $$
 \xymatrix{
 \Ell_T(X) \ar[r]^-{\psi'} \ar[d]_-= & \bigoplus_{i=1}^{i=m} \cO_{\cM_T} \ar[d]^-{\cdot A}_-\cong \\
 \Ell_T(X) \ar[r]^-\rho & \bigoplus_{i=1}^{i=m} \cO_{\cM_T} }
 $$
It follows that there are isomorphisms 
 $$
 coker(\psi) \cong coker(\psi') \cong coker(\rho_1)
 $$
Then the statement follows by Lemma  \ref{mayervietoris}. 
\end{proof}

The following theorem is one of our main results.
\begin{theorem}
\label{main1}
Let $X$ and $X'$ be good toric varieties, and assume that their equivariant elliptic cohomology are isomorphic
$$
\Ell_T(X) \cong \Ell_T(X') \quad \text{in} \quad Coh(\cM_T)
$$
Then $\Sigma_X$ and $\Sigma_{X'}$ have  same number of maximal cones and there is a bijection $$
\phi: (\Sigma^{n-1}_ X)^* \to (\Sigma^{n-1}_ {X'})^*
$$ 
such that for all $\tau \in (\Sigma^{n-1}_ X)^*$, 
$\langle \tau \rangle = \langle \phi(\tau) \rangle$. 
\end{theorem}
 \begin{proof}
By Lemma \ref{vector bundle}, $\Ell_T(X)$ is a vector bundle of rank $|\Sigma^{top}_X|$ and this implies the first part of the claim. Now, by Lemma \ref{main}, if $\Ell_T(X)$ is isomorphic to $\Ell_T(X')$  there is an isomorphism
\begin{equation}
\label{generic}
\bigoplus_{\tau \in (\Sigma^{n-1}_X)^*} \cO_{\cM_{\tau}} \cong \bigoplus_{\tau' \in (\Sigma^{n-1}_{X'})^*} \cO_{\cM_{\tau'}}
\end{equation}
away from a subscheme of codimension 2. Restricting isomorphism (\ref{generic}) to the generic points of the smooth divisors in $\cM_T$ yields a bijection 
$$
\phi: \Sigma^{n-1}_{X'} \to \Sigma^{n-1}_{X}$$
with the property that for every $\tau \in  \Sigma^{n-1}_{X'}$, the divisors 
$\cM_\tau$ and $\cM_{\phi(\tau)}$ are equal. We conclude because of the evident double implication
$$
 \cM_\tau = \cM_{\phi(\tau)} \Longleftrightarrow \langle \tau \rangle = 
\langle \phi(\tau) \rangle  
$$
 \end{proof}
 
 \subsection{The case of toric surfaces}
Let $X$ and $X'$ be good toric varieties. Theorem \ref{main1} provides necessary conditions for $\Ell_T(X)$ and $\Ell_T(X')$ to be isomorphic. We do not expect that  these conditions are sufficient, in general. However it is easy to see that they are sufficient  if $X$ and $X'$ are toric surfaces. We   prove this  in this section. 
 
 Let $X$ be a good toric surface. We fix a $1$-dimensional cone $\tau$ of $\Sigma_X$ and number clock-wise the remaining one-dimensional cones starting from $\tau_1:=\tau$, $$
\Sigma^1_X=\{\tau_1, \ldots, \tau_m\}
$$ 
We also number from $1$ to $N$  he maximal cones of 
$\Sigma$ 
$$\Sigma_X^{top}=\{\sigma_1, \ldots, \sigma_m\}$$ 
We do not need to impose any compatibility between the orderings of $\Sigma_X^1$ and $\Sigma_X^{top}$.

Let $A_X \in M_m(\mathbb{Z})$ be the (weighted) incidence matrix  defined as follows
\begin{itemize}
\item $a_{i, j} = 1$ if the following two conditions are satisfied: $\tau_j \subset \sigma_i$ and if $\tau_j \subset \sigma_{i'}$ then $i \leq i'$
\item if the following two conditions are satisfied: $\tau_j \subset \sigma_i$ and if $\tau_j \subset \sigma_{i'}$ then $i \geq i'$
\item $a_{i, j} = 0$ otherwise. 
\end{itemize}
Recall that we have  $\cO_{E_{\sigma}} =\cO_{E^2}$ for all maximal cones $\sigma$. 
Consider the following maps of coherent sheaves on $\cM_T$ 
$$
\bigoplus_{i=1}^{m} \cO_{\cM_T}  \stackrel{A_X} \longrightarrow 
\bigoplus_{i=1}^{m} \cO_{\cM_T}  \longrightarrow \bigoplus_{i=1}^{m} \cO_{\cM_{\tau_i}}
$$
where the second map is the direct sum of the quotient morphisms 
$\, 
\cO_{\cM_T} \to \cO_{\cM_{\tau_i}}
$. 
Let $\rho_X$ be the composite map. Then $\rho_X$ is equal to the map $\rho_2$ considered in Lemma \ref{mayervietoris}. Thus, by Lemma \ref{mayervietoris}, the kernel of $\rho_X$ is  isomorphic to the equivariant elliptic cohomology of $X$
\begin{equation}
0 \to \Ell_T(X) \to \bigoplus_{i=1}^{i=m} 
\cO_{\cM_T}  
\stackrel{\rho_X} \longrightarrow 
\bigoplus_{i=1}^{i=m} 
\cO_{\cM_{\tau_i}}
\end{equation}

\begin{proposition}
\label{reconstrucsurf}
Let $X$ and $X'$ be two good toric surfaces. Assume that there is a bijection $ 
\alpha: \Sigma^1_X \to \Sigma^1_{X'}$ such that for all $\tau \in \Sigma^1_X$, $\langle \tau \rangle = \langle \phi(\tau) \rangle$. Then there is an isomorphism $\Ell_T(X) \cong \Ell_T(X')$.
\end{proposition}
\begin{proof}
By induction, it is enough to assume that the fans $\Sigma_X$ and $\Sigma'_X$ differ by one ray pointing in opposite direction. We assume without loss of generality that $\Sigma_X$ and $\Sigma '_X$ differ by the ray $\tau_1$. 
That is, the ordered sets of 1-dimensional cones in the fans are given by
$$
\Sigma^1_X=\{ \tau_1, \tau_2, \ldots, \tau_m\}  \quad \Sigma^1_{X'} = \{ \tau_2, \ldots, \tau_i, -\tau_1, \tau_{i+1} , \cdots, \tau_m \}$$ 
It is now easy to see that the incidence matrix 
$A_{X'}$ can be obtained from  $A_{X}$ by row transformations. In particular, there is an invertible $m \times m$ matrix $M$  such that $A_{X} =A_{X'} M$. Let us denote the rows of $A_{X}$ by $r_j$ and the rows of $A_{X'}$ by $r_j'$. The matrix $M$ encodes the following row-transformations 
\begin{itemize}
\item $r_1'=r_1+r_2, r'_j=r_{j+1}$ for $2\le j \le i-1 $
\item $r_i'= r_1+r_{i+1}+\ldots+ r_N$
\item $r_{i+1}'= -r_1+r_{i+2} +\ldots+ r_N$ 
\item $r_j'=r_j$ for $j \ge i+2$
\end{itemize}

 As a consequence the diagram 
$$
\xymatrix{
\bigoplus_{i=1}^{m} \cO_{\cM_T} \ar@/^2.0pc/[rr]^-{\rho_X} \ar[r]^{A_X } \ar[d]^{M \cdot} &   
\bigoplus_{j=1}^{m} \cO_{\cM_T}  \ar[r] \ar[d]^-{=} & \bigoplus_{j=1}^{m} \cO_{\cM_{\tau_j }}\ar[d]^-{=}
\\
\bigoplus_{i=1}^{m} \cO_{\cM_T} 
\ar@/_2.0pc/[rr]^-{\rho_{X'}}
 \ar[r]^{A_{X'}} &   
\bigoplus_{j=1}^{m} \cO_{\cM_T}  \ar[r] & \bigoplus_{j=1}^{m} \cO_{\cM_{\tau_j}} 
}
$$ 
commutes, and therefore the kernels of the horizontal compositions $\rho_X$ and $\rho_{X'}$ are isomorphic. As these kernels compute elliptic cohomology, this concludes the proof.
\end{proof}

\begin{example}
\label{toricsurf}
Let us give an example of two smooth and proper toric surfaces $X$ and 
$X'$ which are not isomorphic as varieties, but have isomorphic equivariant elliptic cohomology. This shows in particular that $\Ell_T(X)$ as coherent sheaf over $\cM_T$ is not by itself  sufficient to reconstruct $X$, nor its full GKM graph. We define $\Sigma_X$ and $\Sigma_{X'}$ by plotting the generators of their one-dimensional cones below 
 $$
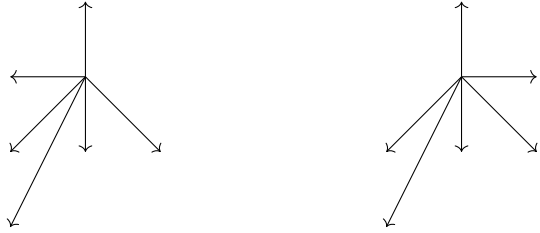
\begin{tikzpicture} 
\draw[->] (5,0) -- (6,-1); 
\draw[->] (5,0) -- (5,1); 
\draw[->] (5,0) -- (5,-1); 
\draw[->] (5,0) -- (4,0 ); 
\draw[->] (5,0) -- (4,-1 ); 
\draw[->] (5,0) -- (4,-2 );

\draw[->] (10,0) -- (11,0); 
\draw[->] (10,0) -- (10,1); 
\draw[->] (10,0) -- (11,-1); 
\draw[->] (10,0) -- (9,-1 );
\draw[->] (10,0) -- (10,-1 );
\draw[->] (10,0) -- (9,-2 );

\end{tikzpicture}
$$
Now, the equivariant elliptic cohomology of $X$ and $X'$ are isomorphic by Proposition \ref{reconstrucsurf}. However $X$ and $X'$ are  non isomorphic as varieties, even forgetting the $T$-action. This can be easily checked directly.  One can also note that the fans of toric varieties that are isomorphic as abstract varieties must  differ by an automorphism of $N_T$, see for instance \cite{68568} for references and dicussions; and no automorphism of $N_T$ sends $\Sigma_X$ to $\Sigma_{X'}$.
 \end{example}

   \section{A coherent cohomology calculation}
 \label{sccc}
In this section we will prove the following result.
\begin{theorem}
\label{coho}
Let $X$ be a good toric variety. The dimension of the cohomology group $H^n(\mathcal{O}_{\cM_X})$ is equal to the number of irreducible components of $\cM_X$. Equivalently, it is equal to the number of top dimensional cones in $\Sigma_X$
$$
\mathrm{dim}_{\mathbb{C}}(H^n(\mathcal{O}_{\cM_X})) = |\Sigma_X^{top}|
$$
\end{theorem}

The trouble with computing the coherent cohomology of $\cM_X$ is that in general   $\cM_X$ is  not normal crossing. In fact it is easy to see that, if $X$ is proper, then  $\cM_X$ is normal crossing if and only if $X$ is isomorphic to the projective space. Most results on coherent cohomology of singular reducible schemes in the literature (such as \cite{bakhtary2010cohomology}) work under the  normal crossing assumption.  In that setting  theorems for coherent cohomology are available which would imply in particular Theorem \ref{coho}.

Thus, to prove Theorem \ref{coho}, we cannot appeal to general results. Instead, we will perform an explicit \v{C}ech calculation via elementary methods. We explain our strategy first in the case of $E^n$. We build an especially small and manageable complex which computes coherent \v{C}ech cohomology of $E^n$. The properties we will need are summed up in Lemma \ref{smallcomplex}. Then  we pass to the general case. 

\subsection{The coherent cohomology of $E^n$}
\label{sec:tccoe}
Let $e$ be the identity element in $E$, and choose a second closed point $p$  in $E$, $p \neq e$. Let $U_a=E - \{p\}$, $U_b=E-\{e\}$, and $U_c=U_a \cap U_b$. We equip the set $\{a, b, c\}$ with the partial order $c < a$ and  $c < b$. For every $i=(i_1, \ldots, i_n) \in I:=\{a, b\}^n$, we denote by $U_i$  the open subset 
$$
U_i:=  U_{i_0} \times \ldots \times U_{i_n} \subset E^n
$$
We can compute the coherent cohomology of $E^n$  via  the affine open cover  $\mathfrak{U}=\{U_i\}_{i \in \{a, b\}^n}$. Geometrically, this means viewing $E^n$ as the realization of the \v{C}ech nerve of the cover $\mathfrak{U}$. The drawback is that the \v{C}ech complex obtained in this way is rather large. We take a slightly different approach, by expressing $E^n$ as a smaller colimit of affine open subsets.  

Consider the indexing set $J=\{a, b, c\}^n$. For every $j=(j_1, \ldots, j_n) \in J$ we denote 
$$
U_j = U_{j_1} \times \ldots \times U_{j_n} \subset E^n. 
$$
We can equip $J$ with the product partial order. Note that $j$ and $j'$ are in $J$, then $j \leq j'$ if and only if $U_j \subseteq U_{j'}$. The poset $J$ has a least element given by $(c, \ldots, c)$ which corresponds to the open subset 
$U_c \times \ldots \times U_c$. All open subsets $U_j$ can be realized as intersections of elements of $\mathcal{U}$. In particular $U_c \times \ldots \times U_c$ is the intersection of all open subsets in $\mathcal{U}$. 

% For future reference, we make the following definition.
%
%\begin{definition}
%\label{fraks}
%We denote $\mathfrak{S}$ the collection of open subsets of $E^n$ given by $\{U_j\}_{j \in J}$ ordered by inclusion.
%\end{definition}
%The collection  $\mathfrak{S}$ gives us 

We obtain a $J$-indexed diagram in the category of schemes, such that all arrows are inclusions of open subsets.
\begin{lemma}
\label{colimit}
There is an isomorphism
$$\varinjlim_{j \in J} U_j \cong E^n$$
\end{lemma}
\begin{proof}
The proof depends on an induction on $n$. The case $n=1$ is clear: it reduces to the statement that if $X$ is a scheme and $U$ and $V$ are a open subsets covering $X$, then $X$ is isomorphic to the push-out of the diagram
$$
U \leftarrow U\cap V \rightarrow V
$$
%$$
%\xymatrix{
%U \cap V \ar[r] \ar[d] & U \ar[d] \\
%V \ar[r] & X}
%$$
in the category of schemes. 

Let us prove the inductive step. We can break down $J$ into three full subcategories, depending on the letter appearing as the last coordinate. Namely, for every $* \in \{a, b, c\}$ we define $J_*$ to be the full subcategory of $J$ on the vertices of the form  $$
U_{j_1} \times \ldots \times U_{j-1} \times U_{*}. 
$$
By the inductive hypothesis, we have that 
\begin{equation}
\label{indu}
\varinjlim_{j \in J_*} U_j \cong E^{n-1} \times U_*
\end{equation}
As colimits commute with colimits, we can compute $\varinjlim_{j \in J}$ in terms of the colimits of the subdiagrams indexed by $J_*$: more precisely, there is a push-out
$$
\xymatrix{ \varinjlim_{j \in J_c} U_j \ar[r] \ar[d] &\varinjlim_{j \in J_a} U_j \ar[d] \\
\varinjlim_{j \in J_b} U_j \ar[r] & \varinjlim_{j \in J} U_j}
$$
By (\ref{indu}) we can rewrite this as 
$$
\xymatrix{ E^{n-1} \times U_c \ar[r] \ar[d] & E^{n-1} \times U_a  \ar[d] \\
E^{n-1} \times U_b \ar[r] & \varinjlim_{j \in J} U_j}
$$
Now, we are back to the case of a scheme $X=E^n$ covered by two open subsets, $U=E^{n-1}\times U_a$ and $V=E^{n-1}\times U_b$, with intersection 
$E^{n-1}\times U_c$. We know that the push-out has to be isomorphic to $E^n$ and  therefore we conclude that 
$$
\varinjlim_{j \in J} U_j \simeq E^n
$$
as we wished to prove.
\end{proof}
If $X$ is a scheme, we denote by $H^*(\mathcal{O}_X)$ the derived global sections of $\mathcal{O}_X$, viewed as an object of the stable dg category $\mathbb{C}\text{-}\mathrm{mod}$.
\begin{lemma}
\label{h*lim}
There is an equivalence in $k\text{-}\mathrm{mod}$
$$
H^*(\mathcal{O}_{E^n}) \simeq \varprojlim_{j \in J} H^0(\mathcal{O}_{U_j})
$$
\end{lemma}
\begin{proof}
Let $\iota_j:U_j \to E^n$ be the inclusion, and let 
$p:E^n \to \mathrm{pt}$ be the structure map. By Proposition \ref{descent} and Lemma \ref{colimit} there is an equivalence, in the category of symmetric monoidal presentable categories, 
$$
 \Qcoh(E^n) \simeq \varprojlim_{j \in J} \Qcoh(U_j).
$$ 
This implies that there an is equivalence in  
$\Qcoh(E^n)$, 
 $$\mathcal{O}_{E^n} \simeq \varprojlim_{j \in J} \iota_{j*}\mathcal{O}.$$ 
 The statement follows by applying $p_*$ to both sides of the equivalence, 
 $$
H^*(\mathcal{O}_{E^n}) = 
p_* \mathcal{O}_{E^n} \simeq \varprojlim_{j \in J} p_* (\iota_{j*}\mathcal{O}) \simeq \varprojlim_{j \in J} H^*(\mathcal{O}_{U_j}) \simeq 
\varprojlim_{j \in J} H^0(\mathcal{O}_{U_j}).
 $$
 where the last step depends on the fact that $U_j$ is affine and therefore has no higher coherent cohomology.
  \end{proof}
  
  For every $0 \leq k \leq n$, let $J_k \subset J=\{a, b, c\}^n$ be the subset of vectors such that 
  $c$ appears as a coordinate exactly $k$ times.     \begin{lemma}
  \label{smallcomplex}
The object $H^*(\mathcal{O}_{E^n})$ in 
  $k\text{-}\mathrm{mod}$ is equivalent to a complex $P^\bullet$ having the following properties 
  \begin{enumerate}
  \item For every $0 \leq k \leq n$, $P^k \cong  \oplus_{j \in J_k} H^0(\mathcal{O}_{U_j})$
    \item If $k < 0$ or $k>n$, then $P^k=0$
    \item The differential $d^k:P^{k-1} \to P^k$ is given by a sum of restrictions, possibly with signs: more precisely, if $l$ is in $J_{k-1}$ then the restriction of the differential    $$
    d^{k}|_{H^0(\mathcal{O}_{U_l})}: H^0(\mathcal{O}_{U_l}) \to \oplus_{j \in J_k} H^0(\mathcal{O}_{U_j})
    $$
    has as factors the pull-back maps $H^0(\mathcal{O}_{U_l}) \to H^0(\mathcal{O}_{U_j})$, where $U_j \subset U_l$ with appropriate twists by $-1$.
  \end{enumerate}
  \end{lemma}
  \begin{proof}
The statement follows from Lemma \ref{h*lim}. We  adapt the inductive argument from Lemma \ref{colimit}. As the steps are the same, we repeat them here in a somewhat abbreviated form.

The base case $n=1$ is immediate. For the inductive step, we use the fact that by Lemma \ref{h*lim}
$$
H^*(\mathcal{O}_{E^n}) \simeq \varprojlim_{j \in J} H^0(\mathcal{O}_{U_j}).
$$
We compute $\varprojlim_{j \in J} H^0(\mathcal{O}_{U_j})$ as the fiber product
$$
\xymatrix{ \varprojlim_{j \in J} H^0(\mathcal{O}_{U_j})\ar[r] \ar[d] &\varprojlim_{j \in J_a} H^0(\mathcal{O}_{U_j}) \ar[d] \\
\varprojlim_{j \in J_b} H^0(\mathcal{O}_{U_j})  \ar[r] & \varprojlim_{j \in J_c} H^0(\mathcal{O}_{U_j})}
$$
Equivalently, we can express $\varprojlim_{j \in J} H^0(\mathcal{O}_{U_j})$ as a cocone 
\begin{equation}
\label{cocone}
\varprojlim_{j \in J} H^0(\mathcal{O}_{U_j}) \longrightarrow \big [
\varprojlim_{j \in J_a} H^0(\mathcal{O}_{U_j}) \oplus
\varprojlim_{j \in J_b} H^0(\mathcal{O}_{U_j})  \longrightarrow \varprojlim_{j \in J_c} H^0(\mathcal{O}_{U_j}) \big ].
\end{equation}
Now, for every $* \in \{a, b, c\}$ there is an equivalence 
$$\varprojlim_{j \in J_*} H^0(\mathcal{O}_{U_j}) \simeq H^*(\mathcal{O}_{E^{n-1}}) \otimes H^0(\mathcal{O}_{U_*})$$
This enables us to rewrite (\ref{cocone}) as 
$$
\varprojlim_{j \in J} H^0(\mathcal{O}_{U_j}) \to \Big [ \big (H^*(\mathcal{O}_{E^{n-1}}) \otimes H^0(\mathcal{O}_{U_a}) \big ) \oplus
\big ( H^*(\mathcal{O}_{E^{n-1}}) \otimes H^0(\mathcal{O}_{U_b}) \big) \to H^*(\mathcal{O}_{E^{n-1}}) \otimes H^0(\mathcal{O}_{U_c})\Big ]
$$
By the inductive hypothesis $H^*(\mathcal{O}_{E^{n-1}})$ satisfies the statement, and  the same applies to 
$$
\big (H^*(\mathcal{O}_{E^{n-1}}) \otimes H^0(\mathcal{O}_{U_a}) \big ) \oplus
\big ( H^*(\mathcal{O}_{E^{n-1}}) \otimes H^0(\mathcal{O}_{U_b}) \big). 
$$ 
The situation is  different for  $H^*(\mathcal{O}_{E^{n-1}}) \otimes H^0(\mathcal{O}_{U_c})$. Indeed, by the inductive hypothesis $H^*(\mathcal{O}_{E^{n-1}}) \otimes H^0(\mathcal{O}_{U_c})$ is equivalent to a complex whose $k$-th term is a direct sum of factors of the form $H^0(U_j)$ where $j$ is in 
$J_{k+1}$, rather than $J_k$: this is due to the fact that we are tensoring $H^*(\mathcal{O}_{E^{n-1}})$ with $H^0(\mathcal{O}_{U_c})$, and this raises by one the number of entries equal to $c$. We conclude via the usual formula for the cocone of a morphism of complexes. 
\end{proof}
Lemma \ref{smallcomplex} gives us a small model for the coherent cohomology of $E^n$. In the next section we will use this model to compute the top  coherent cohomology of $\cM_X$. Before proceeding, in the next subsection, we   sketch without giving all the details an alternative  and more elementary way to prove Lemma \ref{smallcomplex}.  We chose to present our previous argument  instead as it makes more transparent the  geometry underpinning Lemma \ref{smallcomplex}. 
\subsubsection{A different approach}
Another way to prove Lemma \ref{smallcomplex} is to apply iteratively the following elementary and well-known observation to reduce the size of the \v{C}ech complex.
 \begin{lemma}
 \label{small}
Let $\mathfrak{C} = (C^i, \phi^i)$ be a complex 
$$
\mathfrak{C}= \ldots C^{i-1} \stackrel{\phi^{i-1}} \to C^i \stackrel{\phi^i} \to C^{i+1} \stackrel{\phi^{i+1}} \to \ldots
$$
Consider splittings 
$
C^i \cong B^i \oplus K^i \quad C^{i+1}$ and  $B^{i+1} \oplus K^{i+1}
$, and denote the projections 
$$
B^i \stackrel{p_1}\leftarrow C^i \stackrel{p_2}\rightarrow K^i \quad B^{i+1} \stackrel{p_1}\leftarrow C^{i+1} \stackrel{p_2}\rightarrow K^{i+1}.
$$
Set $\phi^i_1=p_1 \circ \phi^i$ and $\phi^i_2=p_2 \circ \phi^i$. Assume that $\sigma:=\phi^i_2|_{K^i}: K^i \to K^{i+1}$ is an isomorphism. 
Then $\mathfrak{C}$ is quasi-isomorphic to the complex 
$\mathfrak{D} = (D^j, \psi^j)$ 
where 
\begin{itemize}
\item For all $j \notin \{i, i+1\}$, $D^j=C^j$. For all $j \notin \{i-1,i,i+1\}$ 
$\psi^j = \phi^j$. 
\item $D^i=B^i$ and $D^{i+1}=B^{i+1}$
\item $\psi^{i-1} = pr_1 \circ \phi^{i-1}$, $\psi^i = (\phi^i_1)|_{B^i} - (\phi^i_1 \circ \sigma^{-1} \circ \phi^i_2)|_{B^i}$ and $\psi^{i+1}=(\phi^{i+1})|_{B^{i+1}}$
\end{itemize}
\end{lemma}
\begin{proof}
The quasi-isomorphism is given by composing the following maps of complexes, the first one of which is an isomorphism
$$
\xymatrix{
\ldots \ar[r] & C^{i-1} \ar[r]^{\phi^{i-1}}   & C^i \ar[r]^{\phi^{i}}  & C^{i+1} \ar[r]^{\phi^{i+1}} & C^{i+2} \ar[r]    & \ldots \\
\ldots \ar[r] & C^{i-1} \ar[r]^{\xi^{i-1}} \ar[u]^= & B^i\oplus K^i \ar[r]^{\xi^{i}}  \ar[u]^\rho & B^{i+1} \oplus K^{i+1} \ar[r]^-{\xi^{i+1}}   \ar[u]^\lambda & C^{i+2} \ar[r] \ar[u]_= & \ldots \\
\ldots \ar[r] & C^{i-1} \ar[r]^{\psi^{i-1}}  \ar[u] & B^i  \ar[r]^{\psi^{i}}    \ar[u]  & B^{i+1} \ar[r]^{\psi^{i+1}}  \ar[u] & C^{i+2} \ar[r] \ar[u]   & \ldots
}
$$
Let us explain how the middle complex, and the maps $\rho$ and $\lambda$ are defined. We will use the identification $C^i=B^i\oplus K^i$ and $C^{i+1}=B^{i+1} \oplus K^{i+1}$ given by the splitting; thus, we also implicitly view $B^i$ and $K^i$ as subspaces of $C^i$ and similarly for $B^{i+1}$ and $K^{i+1}$. We set 
$$(b, k) \in  B^i\oplus K^i =  C^i  \mapsto \rho((b,k)) = (b, \phi^i_2(b) + k)
$$
$$(b, k) \in  B^{i+1}\oplus K^{i+1} =  C^{i+1}  \mapsto \lambda((b,k)) = (b + \phi^i_1\circ \sigma^{-1}(k), k)$$
The definition of $\lambda$ and $\rho$ has just the purpose of \emph{diagonalizing} $\phi^i$. The differentials $\xi^j$ of the middle complex are the unique maps making the diagram commute: namely, $\xi^j=\phi^j$ if 
$j \notin \{i-1, i, i+1\}$; for the other indices we have that 
$$\xi^{i-1}=\rho^{-1} \circ \phi^{i-1}, \quad \xi^i(b, k)= \big ( \phi^i_1(b) - \phi^i_1 \circ \sigma^{-1} \circ \phi^i_2(b), k \big ), \quad \xi^{i+1}=\phi^{i+1} \circ \lambda
$$
\end{proof}

Rather than giving a proof of the fact that Lemma \ref{small} implies Lemma \ref{smallcomplex}, we show in some detail why this works in the case of $E^2$. We consider the open cover 
\begin{equation}
\label{E2cover}
\mathcal{U}=\{U_{aa}, U_{ab}, U_{ba}, U_{bb}\}
\end{equation}
which we introduced in Section \ref{sec:tccoe}.  To compute the \v{C}ech complex, we order the members of the cover as they appear in (\ref{E2cover}).

We obtain the complex below, where one has take the direct sum of all terms on the same row (but for clarity we have dropped the $\oplus$-signs); and the differential is given by the sum of restrictions (which are indicated by arrows in the diagram) twisted by appropriate signs. Note also that instead of writing intersections out explicitly,  we only indicate the resulting open subset: for instance, we write 
$H^0(\mathcal{O}_{U_{ac}})$ instead of $H^0(\mathcal{O}_{U_{aa}\cap U_{ab}})$. 
$$
\xymatrix{&  H^0(\mathcal{O}_{U_{aa}}) \ar[ld] \ar[d] \ar[dr] & H^0(\mathcal{O}_{U_{ab}}) \ar[dll] \ar[dr] \ar[drr] & H^0(\mathcal{O}_{U_{ba}})  
\ar[dll] \ar[d] \ar[drr] & H^0(\mathcal{O}_{U_{bb}})  \ar[dll] \ar[d] \ar[dr]& &
\\
H^0(\mathcal{O}_{U_{ac}}) \ar[dr]  \ar[drr] & H^0(\mathcal{O}_{U_{ca}})  \ar[d] \ar[drr] & \textcolor{red}{H^0(\mathcal{O}_{U_{cc}})}    \ar@[red][d]  \ar[dr] & \textcolor{blue}{H^0(\mathcal{O}_{U_{cc}})} \ar@[blue][dll] \ar[dr] & 
H^0(\mathcal{O}_{U_{cb}})  \ar[dll] \ar[d] &  H^0(\mathcal{O}_{U_{bc}}) \ar[dl] \ar[dll] \\
& \textcolor{blue}{H^0(\mathcal{O}_{U_{cc}})}
 \ar[drr] & \textcolor{red}{H^0(\mathcal{O}_{U_{cc}})} \ar[dr] & \textcolor{green}{H^0(\mathcal{O}_{U_{cc}})} \ar@[green][d] & H^0(\mathcal{O}_{U_{cc}}) \ar[dl] & \\
&&&  \textcolor{green}{H^0(\mathcal{O}_{U_{cc}})} &&&
}
$$
The pairs whose colours match cancel out by applying Lemma \ref{small} to the arrow joining them. 
%It is easy to write down a simple algorithm which prescribes the order in which cancelations are to be performed: the one we followed gives the sequence red, blue, and green; but different sequences of steps are also possible. 
We  apply Lemma \ref{small} a maximal number of times: in this way we remove all pairs such that the restriction is an isomorphism. Note that each iteration of Lemma \ref{small} also affects the differential. Ultimately we obtain a complex that satisfies all the properties of Lemma \ref{smallcomplex}. We leave the details to the reader. 

\subsection{The coherent cohomology of $\cM_X$} Let us return to the setting of Section \ref{sccc}. Let  $X$ be a good $n$-dimensional toric variety with fan $\Sigma_X$. We choose an ordering of  the $1$-dimensional cones of $\Sigma_X$. If $\sigma$ is a top dimensional cone, we consider  
the set of one dimensional cones in its closure 
$$\{v_{i_1}, \ldots, v_{i_n}\}, \quad i_1 < \ldots < i_n$$  There exists a unique lattice automorphism $\phi_\sigma$ of $N_T$ which maps 
$v_{i_j}$ to the $j$-th element of the standard basis
$$
v_{i_1} \mapsto (1, 0, \ldots) \quad 
v_{i_2} \mapsto (0, 1, \ldots) \quad \ldots
$$
Let $\tau_1 \ldots \tau_n$ be the codimension-one cones in the closure of $\sigma$, which we have numbered in such a way that $v_{i_j} \notin \tau_j$. Then $\phi_\sigma$ maps $\sigma$ to the first  quadrant, and maps the sublattices generated by $\tau_i$ to the coordinate hyperplanes
$$
\langle \tau_1 \rangle \mapsto \langle (0,1,0 \ldots), (0,0,1 \ldots), \ldots \rangle  \quad 
\langle \tau_2 \rangle \mapsto \langle (1,0,0 \ldots), (0,1,0 \ldots) \ldots \rangle \quad \ldots
$$
Note that $\phi_\sigma$ induces an isomorphism $E \otimes \langle \sigma \rangle \cong E^n$, which maps the divisors $E \otimes \langle \tau_i \rangle$ to the standard divisors 
$$
\{e\} \times E^{n-1}, \quad E \times \{e\} \times E^{n-2}, \quad \ldots
$$

Let $\pi: \widetilde{\cM_X} \to \cM_X$ be the normalization. The scheme $\widetilde{\cM_X}$  is a disjoint union of irreducible components which are in natural bijection with the top dimensional  cones of $\Sigma_X$
$$
\widetilde{\cM_X} \cong \coprod_{\sigma \in \Sigma_X^{top}} \cM_\sigma
$$ 
The component $\cM_\sigma$ corresponding to the  top dimensional cone $\sigma$ is canonically isomorphic to $E \otimes \langle \sigma \rangle$. We can identify $\cM_\sigma$ with $E^n$,  via the isomorphisms $\phi_\sigma$ constructed above: with small abuse of notation we denote $\phi_\sigma$ also the resulting   isomorphism  $\cM_\sigma \cong E^n$. Using $\phi_\sigma$ we will be able to define open subsets of $\cM_\sigma$ in terms of  the  open cover $\mathfrak{U}$ of $E^n$ which we considered in Section \ref{sec:tccoe}.
 
\begin{lemma}
\label{tiaua}
There is a unique affine open cover $\mathfrak{U}_X$ of $\cM_X$ having the following property: 
\begin{itemize}
\item For every top dimensional cone  $\sigma \in \Sigma^n_X$ the open cover of $\cM_\sigma$ given by 
$$
\big \{ \pi^{-1}(U) \cap \cM_\sigma \big \}_{U \in \mathfrak{U}_X}
$$ coincides with the standard open cover $\mathfrak{U}$ of $E^n$ under the identification $\cM_\sigma \cong E^n$.
\end{itemize}
\end{lemma}
\begin{proof}
On $\widetilde{\cM_X}$ there is an obvious open cover, which is given by considering the open cover $\mathfrak{U}$ on each connected component $\cM_\sigma \cong E^n$. This however does not descend to an open cover on $\cM_X$. Note, indeed, that the normalization map $\pi: \widetilde{\cM_X} \to \cM_X$ is not open. The issue is that the image under $\pi$ of an open subset $U  \subset \cM_\sigma$  might intersect the singular locus of 
$\cM_X$ (and this happens also for the open subsets in the cover $\mathfrak{U}$). 
The singular locus is where different irreducible components come together: so an open subset containing a point $p$ of  the singular locus must necessarily restrict to a non-empty open subset on all irreducible components of $\cM_X$  meeting at $p$.

The open cover $\mathfrak{U}_X$ is the best approximation to the open cover on $\widetilde{\cM_X}$ induced by $\mathfrak{U}$. We explain how to generate the open subsets belonging to $\mathfrak{U}_X$ via an iterative process. The recipe we will describe starts from an open subset $U \in \mathfrak{U}$ on a component $\cM_\sigma$, and generates an open subset of $\cM_X$. The open cover $\mathfrak{U}_X$ is made up of all the open subsets of $\cM_X$ that can be obtained in this way. We refer the reader  to Example \ref{procedure}, after the proof, for some concrete examples.
% For convenience we may choose an ordering of the connected components of $\widetilde{\cM_X}$, and of the open subsets in $\mathfrak{U}$: then we apply our procedure in a sequential way from the first open subset if $\widetilde{\cM_X}$ in the list to the last one. At each step, if a new open subset has been generated, we add it to $\mathfrak{U}_X$; otherwise we pass directly to the next open subset in the list.

Let $U \in \mathfrak{U}$ be an open subset of a component $\cM_\sigma$. Note that $U$ is a product of open subsets of $E$ of the form $U_a$ and $U_b$, where we are using the identification $\cM_\sigma \cong E^n$ provided by $\phi_\sigma$. A codimension-one cone $\tau$ of $\sigma$ corresponds to a divisor $\cM_\tau$ in $\cM_\sigma$. We associate to $U$ the collection of codimension-one cones in $\sigma$ 
$$
\tau_{i_1},  \ldots, \tau_{i_k}
$$
having the property that $U \cap \cM_{\tau_i} \neq \varnothing$. It is easy to see that $k$ is the number of factors of $U$ of the form $U_a$. Consider the codimension $k$ cone $\rho$ given by the intersection $\tau_{i_1} \cap \ldots \cap \tau_{i_k}$. Let $\sigma_{j_1}, \ldots, \sigma_{j_n}$ be the collection of top dimensional cones of $\Sigma$ containing $\rho$. Note that there is a   fixed isomorphism between $\cM_\sigma$ and  the components $\cM_{\sigma_{j_l}}$ given by 
$$
\psi_{\sigma, \sigma_{j_l}}: \cM_\sigma \stackrel{\phi_\sigma} \longrightarrow E^n  \stackrel{(\phi_{\sigma_{j_l}})^{-1}} \longrightarrow \cM_{\sigma_{j_l}}
$$
Now consider the open subset of $\widetilde{\cM_X}$ defined as 
$$
V:= U  \bigcup \big ( \bigcup_{i=1, \ldots, n} \psi_{\sigma, \sigma_{j_l}}(U) \big )
$$

The claim is that $\pi(V)$ is an open subset of $\cM_X$. According to our recipe, we need to add it to the affine open cover $\mathfrak{U}_X$. To construct $\mathfrak{U}_X$ we iterate this procedure on every component $\cM_\sigma$ of $\widetilde{\cM_X}$, and on every open subset $U \in \mathfrak{U}$ of $\cM_\sigma$. Of course  not all iterations will generate a new element of $\mathfrak{U}_X$. Let us conclude by making a remark on one aspect of this construction: why can we conclude that $\pi(V)$ is affine, from the fact that its irreducible components are affine? The point is that 
$\pi(V)$ is an iterated push-out, in the category of schemes, of affine schemes along closed immersions:  and therefore it is affine by Theorem 3.4 of \cite{schwede2005gluing}.  It is easy to see that $\mathfrak{U}_X$ satisfies the property stated in the Lemma, and that that property pins it down uniquely, and this concludes the proof.
\end{proof}

\begin{example}
\label{procedure}
Let us explain some special cases of the constructions of Lemma \ref{tiaua}. 
\begin{itemize}
\item Consider the open subset $U=(U_a)^n$ contained in a connected component $\cM_\sigma$. As all factors of 
$U$ are of the form $U_a$, we associate to it the collection of all codimension-one cones in $\sigma$. Their intersection is the zero cone, which is contained in all top dimensional cones of $\Sigma$. Thus the open subset in 
$\mathfrak{U}_X$ obtained from $U$ via  the recipe of Lemma \ref{tiaua}  has non-trivial intersection with all the irreducible components of $\cM_X$. At the other extreme, the open subset $U=(U_b)^n \subset \cM_\sigma$ does not contain $U_a$-factors, and so the corresponding open subset in $\mathfrak{U}_X$ is just $\pi(U)$, which lies entirely in one single irreducible component of $\cM_X$. Note that $\pi$ restricted to $U$ is an open map, as $U$ does not intersect the preimage of the singular locus of $\cM_X$.

\item Let $X=\mathbb{P}^1$. Then  $\cM_X$ is given by the push-out of the diagram
$$
E \longleftarrow \{e\} \longrightarrow E
$$
where $e$ is the identity element. For convenience let us denote by $E_1$ the first copy of $E$ appearing in the diagram, and by $E_2$ the second copy. Then $\mathfrak{U}_X$ contains three open subsets: 
\begin{enumerate}
\item $E_1-\{e\}$
\item $E_2 - \{e\}$
\item the push-out of the following diagram 
$$
E_1-\{p\} \longleftarrow \{e\} \longrightarrow E_2-\{p\}
$$
\end{enumerate}
\item Let us describe  briefly the case of $X=\mathbb{P}^2$. Then $\cM_X$ has three irreducible components. In this case $\mathfrak{U}_X$ contains seven open subsets. 
\begin{enumerate}
\item There is a unique open subset with three irreducible components. This corresponds to the case considered in the first point in this list of examples, where we start with an open subset of the form $(U_a)^3$ on one of the components of $\widetilde{\cM_X}$. 
\item There are three open subsets with two irreducible components. They are generated from  open subsets $U$ of the form $U_a \times U_b$, or $U_b \times U_a$, on a component $\cM_\sigma$. There is only one codimension-one cone $\tau$ associated with $U$, and therefore two maximal cones $\sigma$ and $\sigma'$ such that $\tau \subset \sigma$ and 
$\tau \subset \sigma'$. 
\item  There are three open subsets with one irreducible component. They correspond to the open subsets of the form $U_b \times U_b$ on the three components of  $\widetilde{\cM_X}$. 
\end{enumerate}
\end{itemize}
\end{example}

As in Section \ref{sec:tccoe}, we will not actually  use the \v{C}ech nerve of $\mathfrak{U}_X$ for our calculations. Rather, we will work with a smaller colimit  giving a presentation of 
$\cM_X$. Recall the ordered collection of open subsets 
$\{U_j\}_{j \in J}$ of $E^n$   which we studied in Section \ref{sec:tccoe}. 
\begin{lemma}
\label{tiauc2}
There is a unique collection $\{U_j\}_{j \in J_X}$ of affine open subsets of $\cM_X$ ordered by inclusions having the following property:
\begin{itemize}
\item For every top dimensional cone  $\sigma \in \Sigma^n_X$ the collection  
$$
\big \{ \pi^{-1}(U_j) \cap \cM_\sigma \big \}_{j \in J_X}
$$ coincides with the collection of open subsets 
$\{U_j\}_{j \in J}$ under the identification 
$$\phi_\sigma:\cM_\sigma \cong E^n$$
\end{itemize}
\end{lemma}
\begin{proof}
The statement is a small variation on Lemma \ref{tiaua} and the proof goes along the same lines. Namely we can set up a recursion where start from an open subset $U_j$ with $j\in J$ on a component $\cM_\sigma$ and, if it has non-trivial intersection with the preimage of the singular locus of $\cM_X$, we use the isomorphisms $\phi_\sigma$ to generate an actual open subset in $\cM_X$.

Before moving on however let us make a simple  observation. The open subsets $U_j$ with $j \in J_X$ all arise as intersections of open subsets in $\mathfrak{U}_X$. In fact, we can say something more precise. Assume that $U$ is an open subset in $\mathfrak{U}$, or more generally in $\{U_j\}_{j \in J}$, on a component $\cM_\sigma$. For the remainder of this argument, we will denote $\phi(U)$ the element of $\mathfrak{U}_X$ (or more generally of $\{U_j\}_{j \in J_X}$) which is  generated from $U$ via the procedure described in Lemma \ref{tiaua} and that, as we explained, also applies to the current Lemma \ref{tiauc2}. Now let $U$ be an open subset in $\{U_j\}_{j \in J}$, on a component $\cM_\sigma$. As we know, $U$ can be written as an intersection of open subsets $V_1 \ldots V_k$ in $\mathfrak{U}$, 
$
U=V_1 \cap \ldots \cap V_k 
$. 
It is easy to see that $\phi(U)$ can be written as the intersection
$$
\phi(U) = \phi(V_1) \cap \ldots \cap \phi(V_k). 
$$
\end{proof}
For every $j \in J_X$, we denote $\widetilde{U_j}:=\pi^{-1}(U_j)$. The open subset $\widetilde{U_j}$ is affine, and decomposes as a disjoint union of open subsets lying on various components of $\widetilde{\cM_X}$.  We obtain two diagrams indexed by $J_X$ in the category of schemes, namely $\{U_j\}_{j \in J_X}$ and $\{\widetilde{U_j}\}_{j \in J_X}$. In either case, the vertices are affine schemes, and the arrows are open inclusions.
\begin{lemma}
\label{colimit2}
There are isomorphisms
\begin{enumerate}
\item $\varinjlim_{j \in J_X} \widetilde{U_j} \cong \widetilde{\cM_X}$ 
\item $\varinjlim_{j \in J_X} U_j \cong \cM_X$
\end{enumerate}
\end{lemma}
\begin{proof}
The first statement follows immediately from \ref{colimit2}. Indeed $\widetilde{\cM_X}$ is isomorphic to a disjoint union of copies of $E^n$
$$
\widetilde{\cM_X} \cong \coprod_{\sigma \in \Sigma^n_X} E^n
$$
On the other hand, each element of $\{\widetilde{U_j}\}_{j \in J_X}$ is a disjoint union of open subsets in $\{U_j\}_{j \in J}$ lying on different connected components of $\widetilde{\cM_X}$. This readily implies equivalence $(1)$, as disjoint unions are colimits  and therefore commute with colimits.  

The second claim is proved in a similar way. By Corollary \ref{pushouttoric} $\cM_X$ is isomorphic to the colimit 
$$
 \big [ \coprod_{\tau \in (\Sigma^{n-1}_X)^*} \cM_{\tau} \rightrightarrows \coprod_{k \in I} \cM_{\sigma_k} \big ] \to \cM_X
$$
where $I=\{1 \ldots m\}$ is the chosen indexing set for the top dimensional cones of $\Sigma_X$. 
By the first part of the claim we can replace  
$
\coprod_{\sigma \in \Sigma^n_X} \cM_{\sigma} \cong \widetilde{\cM_X}
$ 
with the colimit $\varinjlim_{j \in J_X} \widetilde{U_j}$. Next we can rewrite each $U_j$ as the colimit
\begin{equation}
\label{colimitopens}
\Big [ \big ( \coprod_{\tau \in (\Sigma^{n-1}_X)^*} \cM_{\tau} \big ) \bigcap \widetilde{U_j} \rightrightarrows \widetilde{U_j} \Big ] \to U_j
\end{equation}
Again, we can conclude because colimits commute with colimits.
\end{proof}

The open subsets $\widetilde{U_j}$ have the property that for every $\sigma \in \Sigma^n_X$, the intersection of 
$\widetilde{U_j}$ with a component $\cM_\sigma$ decomposes as a product of $U_a$, $U_b$ and 
$U_c$; further the number of factors equal to $U_a$, $U_b$ or $U_c$ stays always the same, independently of $\sigma$. This observation allows us to meaningfully define, for every $0 \leq k \leq n$, the sub-poset $J_{X,k} \subset J_X$ to be the subset of indices such that all connected components of 
  $\widetilde{U_j}$ decompose as a product where the factor  
  $U_c$  appears exactly $k$ times.     
 
Lemma \ref{smallcomplex2} below is a generalization of 
  Lemma \ref{smallcomplex}. We split it into two parallel statements, one  for $\widetilde{\cM_X}$ and one for $\cM_X$; a  third claim  gives an explicit description of the pull-back  along $\pi$ in coherent cohomology.

  \begin{lemma}
  \label{smallcomplex2}
Consider the derived global sections $H^*(\mathcal{O}_{\widetilde{\cM_X}})$ and 
$H^*(\mathcal{O}_{\cM_X})$. 
\begin{enumerate}
\item The object $H^*(\mathcal{O}_{\widetilde{\cM_X}})$ in 
  $k\text{-}\mathrm{mod}$ is equivalent to a complex $Q^\bullet$ having the following properties 
  \begin{enumerate}
  \item For every $0 \leq k \leq n$, $Q^k \cong  \oplus_{j \in J_{X,k}} H^0(\mathcal{O}_{\widetilde{U_j}})$
    \item If $k < 0$ or $k>n$, then $Q^k=0$
    \item The differential $d^k:Q^{k-1} \to Q^k$ is given by a sum of restrictions, possibly with signs: more precisely, if $l$ is in $J_{X, k-1}$ then the restriction of the differential    $$
    d^{k}|_{H^0(\mathcal{O}_{\widetilde{U_l}})}: H^0(\mathcal{O}_{\widetilde{U_l}}) \to \oplus_{j \in J_{X,k}} H^0(\mathcal{O}_{\widetilde{U_j}})
    $$
    has as factors the pull-back maps $H^0(\mathcal{O}_{\widetilde{U_l}}) \to H^0(\mathcal{O}_{\widetilde{U_j}})$, where $\widetilde{U_j} \subset \widetilde{U_l}$ with appropriate twists by $-1$.
  \end{enumerate}
  \item The object $H^*(\mathcal{O}_{\cM_X})$ in 
  $k\text{-}\mathrm{mod}$ is equivalent to a complex $R^\bullet$ having the following properties 
  \begin{enumerate}
  \item For every $0 \leq k \leq n$, $R^k \cong  \oplus_{j \in J_{X,k}} H^0(\mathcal{O}_{U_j})$
    \item If $k < 0$ or $k>n$, then $R^k=0$
    \item The differential $d^k:R^{k-1} \to R^k$ is given by a sum of restrictions, possibly with signs: more precisely, if $l$ is in $J_{X, k-1}$ then the restriction of the differential    $$
    d^{k}|_{H^0(\mathcal{O}_{U_l})}: H^0(\mathcal{O}_{U_l}) \to \oplus_{j \in J_{X,k}} H^0(\mathcal{O}_{U_j})
    $$
    has as factors the pull-back maps $H^0(\mathcal{O}_{U_l}) \to H^0(\mathcal{O}_{U_j})$, where $U_j \subset U_l$ with appropriate twists by $-1$.
  \end{enumerate}
  \item The pull-back morphism $\pi^*:H^*(\mathcal{O}_{\cM_X}) \to H^*(\mathcal{O}_{\widetilde{\cM_X}})$ is equivalent to the morphism of complexes $(\pi^*)^\bullet: R^\bullet \to Q^\bullet$ whose $k$-th component  
   $$
(\pi^*)^k: \oplus_{j \in J_{X,k}} H^0(\mathcal{O}_{U_j}) \to 
  \oplus_{j \in J_{X,k}} H^0(\mathcal{O}_{\widetilde{U_j}})
  $$
is the product of the pull-back maps $H^0(\mathcal{O}_{U_j}) \to 
 H^0(\mathcal{O}_{\widetilde{U_j}})$.
  \end{enumerate}
  \end{lemma}

\begin{proof}
Recall that $\widetilde{\cM_X}$ is isomorphic to $\coprod_{\sigma \in \Sigma^n_X} E^n$, and therefore  there is an equivalence 
\begin{equation}
\label{directdirect}
H^*(\mathcal{O}_{\widetilde{\cM_X}}) \simeq \bigoplus_{\sigma \in \Sigma^n_X} H^*(\mathcal{O}_{E^n})
\end{equation}
Thus the first claim can be proved by taking as starting point  the colimit $\varinjlim_{j \in J_X} \widetilde{U_j} \cong \widetilde{\cM_X}$, and then proceed exactly as in the proof of Lemma \ref{smallcomplex}. Alternatively, we could  deduce claim $(1)$ directly from Lemma  \ref{smallcomplex} and the splitting (\ref{directdirect}). 

Let us discuss the second claim. Again several different  strategies are possible. For instance, much as in the previous case, one can use claim $(2)$ of Lemma \ref{colimit2} and then adapt the proof of Lemma \ref{smallcomplex}. Let us sketch a slightly different argument, which is very similar to the proof of Lemma \ref{colimit2}.

We  express $\cM_X$ as the colimit 
   $$
 \big [ \coprod_{\tau \in (\Sigma^{n-1}_X)^*} \cM_{\tau}  \rightrightarrows \coprod_{\sigma \in \Sigma^n_X} \cM_{\sigma} \big ] \to \cM_X
$$
Arguing as in the proof of Lemma \ref{h*lim}, we turn this colimit into a limit of derived global sections
$$
 H^*(\mathcal{O}_{\cM_X}) \to \big [ H^*(\mathcal{O}_{\widetilde{\cM_X}})  
 \rightrightarrows  \prod_{\tau \in (\Sigma^{n-1}_X)^*} H^*(\mathcal{O}_{\cM_{\tau}})  \big ] 
$$
We can replace both $H^*(\mathcal{O}_{\widetilde{\cM_X}})$ and $H^*(\mathcal{O}_{\cM_{\sigma \cap  \sigma'}})$ with the models provided, respectively, by the first part of Lemma \ref{smallcomplex2}, which we have just discussed; and by  Lemma \ref{smallcomplex}:  indeed  $\cM_{\tau}$ is isomorphic to $E^{n-1}$, and thus  Lemma \ref{smallcomplex} applies to it. We obtain the limit
\begin{equation}
\label{limitlimit2}
 H^*(\mathcal{O}_{\cM_X}) \to \big [ Q^\bullet  
 \rightrightarrows  \prod_{\tau \in (\Sigma^{n-1}_X)^*} P^\bullet_{\cM_\tau} \big ] 
\end{equation}
where we write $P^\bullet_{\cM_\tau}$ to indicate that this is the complex computing  $H^*(\mathcal{O}_{\cM_{\tau}})$.

 Following through the quasi-isomorphisms   leading up to formula (\ref{limitlimit2}), one realizes that the two parallel arrows are just a product of pull-back maps of the form 
$$ 
i^*: H^0(U) \to H^0(U \cap \cM_\tau) 
$$ 
where $U \subset \cM_\sigma$ is an affine open subset, $\cM_\tau \subset \cM_\sigma$ is a divisor, and $U \cong \big ( U \cap \cM_\tau \big ) \times V$ where $V$ is an affine open subset of $E$. In particular, they are degree-wise surjective maps. Under this assumption, the homotopy limit coincides with the classical limit in the abelian category of complexes. As  $H^*(\mathcal{O}_{\cM_X})$ is the limit of the diagram, it is  equivalent to a complex 
$R^\bullet$ whose components are given by degree-wise equalizers: we claim that $R^\bullet$ has all the properties required by the second part of the Lemma. 

Let us explain why this is the case. The degree-wise equalizers are direct sums of diagrams of the form  
\begin{equation}
\label{limitopens}
\Big [ H^0(\mathcal{O}_{\widetilde{U_j}}) 
\rightrightarrows
H^0(\mathcal{O}_{\big ( \coprod_{\tau \in (\Sigma^{n-1}_X)^*} \cM_{\tau} \big ) \bigcap \widetilde{U_j} })  \Big ]  
\end{equation}
where $j$ is in $J_{X,k}$. The equalizer of (\ref{limitopens}) is $H^0(\mathcal{O}_{U_j})$. This can be seen from the colimit  (\ref{colimitopens}),  which we discussed in course of the proof of Lemma \ref{colimit2}. We obtain the direct sum decomposition $$R^k \cong  \bigoplus_{j \in J_{X,k}} H^0(\mathcal{O}_{U_j})$$ for every $0 \leq k \leq n$.

The third part of the Lemma is easily checked, and we leave it to the reader.
\end{proof}

We are now ready to prove Theorem \ref{coho}. 
The main idea of the argument is not difficult, and it can be best understood in the simple example $X=\mathbb{P}^1$. We explain this calculation in Example \ref{examplep1}, after the proof.

\begin{proof}[Proof of Theorem \ref{coho}]
Consider the normalization map 
$\pi: \widetilde{\cM_X} \to \cM_X$. The variety $\widetilde{\cM_X}$ is isomorphic to the disjoint union of $m=|\Sigma^n_X|$ copies of $E^n$. 
Thus  there is  an isomorphism 
$$
H^*(\mathcal{O}_{\widetilde{\cM_X}}) \cong \bigoplus_{\sigma \in \Sigma_X^n} 
H^*(\mathcal{O}_{\cM_\sigma}) \cong 
\bigoplus_{i=1}^{m} H^*(\mathcal{O}_{E^n})
$$
In particular 
$$
\mathrm{dim}(H^n(\mathcal{O}_{\widetilde{\cM_X}})) = |\Sigma^n_X| = m
$$
We will prove that $\pi^*$ induces an isomorphism 
\begin{equation}
\label{cong}
(\pi^*)^n:H^n(\mathcal{O}_{\widetilde{\cM_X}}) \stackrel{\cong}\longrightarrow
H^n(\mathcal{O}_{\cM_X})
\end{equation}
This implies that $H^n(\mathcal{O}_{\cM_X})$ is $m$-dimensional, which is what we want to show.

The model for $\pi^*$ provided by Lemma \ref{smallcomplex2} gives a morphism of complexes of the following shape
$$
\xymatrix{
\ldots \ar[r] &  Q^{n-2}  \ar[r]^-{d_Q^{n-1}} & Q^{n-1} \ar[r]^-{d_Q^{n-1}} & Q^n \ar[r] & 0 \ar[r] & \ldots \\
\ldots \ar[r] \ar[u] &  R^{n-2} \ar[u]^-{(\pi^*)^{n-2}}  \ar[r]_-{d_R^{n-1}} & R^{n-1} \ar[r]_-{d_R^{n}}  \ar[u]^-{(\pi^*)^{n-1}}  & R^n   \ar@{=}[u]^-{(\pi^*)^{n}}  \ar[r] & 0  \ar@{=}[u] \ar[r] & \ldots
}
$$
The degree $n$ map $(\pi^*)^n$ is indicated in the diagram by the identity sign because it is a canonical isomorphism. Indeed if 
$j$ is in $J_{X,n}$ then $U_j \subset \cM_X$ is irreducible and does not intersect the singular locus. Thus the restriction of the normalization 
$\pi:\widetilde{U_j} \to U_j$ is an isomorphism. As the morphism $(\pi^*)^n$ is a product of pull-back maps $\pi^*:H^0(\mathcal{O}_{U_j}) \to H^0(\mathcal{O}_{\widetilde{U_j}})$, it is also a (canonical) isomorphism. 

We will prove that the maps $d^n_Q$ and $d^n_R$ have equal image in $R^n=Q^n$. Let us  break down this claim further. We need to show that for every $j \in J_{X, k-1}$, and every element $\theta \in H^0(\mathcal{O}_{\widetilde{U_j}})$ there is an element 
$$
\xi \in R^{n-1} = \bigoplus_{j \in J_{X,n-1}} H^0(\mathcal{O}_{U_j}) \quad \text{such that} \quad d^n_Q(\theta) = d^n_R(\xi).
$$

We have two possibilities: either $U_j$ is smooth or, which is the same, irreducible; or $U_j$ is singular, and has two irreducible components. In the first case $\widetilde{U_j}$ is  identified with $U_j$ under $\pi$, and is isomorphic to a product of the form
$$
U_c \times \ldots  \times U_b \times \ldots \times U_c
$$
where $U_b$ occurs once, say in $k$-th position, and all the other factors are equal to $U_c$. 
In the second case, $\widetilde{U_j}$ is isomorphic to two copies of an open subset $U$ of the form% the disjoint union of two components which are isomorphic to products of the form 
$$
U=U_c \times \ldots \times U_a \times \ldots \times U_c
$$
where $U_a$ occurs once, say in $l$-th position, and all the other factors are equal to $U_c$.  Denote by $D$ the divisor
$$
D=U_c \times \ldots \times \{e\}  \times \ldots \times U_c  \stackrel{\subset} \longrightarrow U
$$
Then $U_j$ is singular, and is isomorphic to a  pushout of the form 
$$
\xymatrix{
D  \ar[r] \ar[d] &U \ar[d]\\  
U \ar[r] & U_j}
$$

In the first case, there is nothing to prove. Indeed, we have an identification
$$
R^{n-1} \supset H^0(\mathcal{O}_{U_j}) \stackrel{\pi^*} \longrightarrow 
H^0(\mathcal{O}_{\widetilde{U_j}})
$$ 
and so we can set $\xi=\theta$. Let us focus on the second case, and describe in more detail the structure of the restriction of $d^n_R$ and $(\pi^*)^{n-1}$ to 
$H^0(\mathcal{O}_{U_j}) \subset R^{n-1}$. First of all, the map 
$$
d^n_R: H^0(\mathcal{O}_{U_j})  \to R^{n}
$$
has two components. Indeed, $U_j$ has two irreducible components, which we will denote  $V_j$ and $W_j$. There are two distinct open sets $V$ and $W$ which are  isomorphic to $(U_c)^n$ and are subsets of $U_j$  $$
V \rightarrow V_j \subset U_j \supset  W_j \leftarrow W
$$  
As we discussed, $H^0(\mathcal{O}_{V})$ and 
$H^0(\mathcal{O}_{W})$ appear as factors of both $Q^n$ and $R^n$, as $\widetilde{V}=V$ and $\widetilde{W}=W$.

The key point is that $V$ and $W$  are also contained in open subsets $U_{j'}$ and $U_{j''}$  which are both smooth and of the form 
$$
U_c \times \ldots \times U_b \times \ldots \times U_c
$$
where $U_b$ is in $l$-th position, and all other factors are equal to $U_c$. In particular, 
$\widetilde{U_{j'}}=U_{j'}$ and $\widetilde{U_{j''}}=U_{j''}$.

We are now ready to sum up the shape of the relevant maps in the diagram below
$$
\xymatrix{
\textcolor{blue}{R^{n-1}:} & 
\textcolor{blue}{H^0(\mathcal{O}_{U_{j'}})} \ar@[blue][dr]^{\textcolor{blue}{(d^n_R)_{U_j'}^V}} && 
\textcolor{blue}{H^0(\mathcal{O}_{U_{j}})} \ar@[blue][dr]^{\textcolor{blue}{(d^n_R)^W_{U_j}}} \ar@[blue][dl]_{\textcolor{blue}{(d^n_R)^V_{U_j}}}  \ar[ddd]^{\pi^*} && \textcolor{blue}{H^0(\mathcal{O}_{U_{j''}})} \ar@[blue][dl]_{\textcolor{blue}{(d^n_R)^V_{U_{j''}}}} \\
\textcolor{blue}{R^n:} & & \textcolor{blue}{H^0(\mathcal{O}_{V})} && \textcolor{blue}{H^0(\mathcal{O}_{W})} & \\
\textcolor{red}{Q^n:} & & \textcolor{red}{H^0(\mathcal{O}_{\widetilde{V}})} \ar@{=}[u]^{\pi^*} && \textcolor{red}{H^0(\mathcal{O}_{\widetilde{W}})} \ar@{=}[u]_{\pi^*} & \\
\textcolor{red}{Q^{n-1}:} & 
%\textcolor{red}{H^0(\mathcal{O}_{\widetilde{U_{j'}}})} \ar@[red][ur] \ar@{=}[uuu]^{\pi^*} 
&& 
\textcolor{red}{H^0(\mathcal{O}_{\widetilde{U_{j}}}) = H^0(\mathcal{O}_{U})^{\oplus 2}} 
\ar@[red][ur]^{\textcolor{red}{(d^n_Q)_{\widetilde{U_{j}}}^{\widetilde{W}}}} \ar@[red][ul]_{\textcolor{red}{(d^n_Q)_{\widetilde{U_{j}}}^{\widetilde{V}}}} && 
%\textcolor{red}{H^0(\mathcal{O}_{\widetilde{U_{j'}}})} \ar@[red][ul] \ar@{=}[uuu]_{\pi^*} 
}
$$ 
Let us clarify the meaning of the diagram. In blue and red we have reproduced, respectively,  the relevant  portions of  the differentials 
$$
d^n_R:R^{n-1} \to R^n \quad \text{and of} \quad 
d^n_Q:Q^{n-1} \to Q^n
$$
Up to signs, the components of the differentials are just given by restrictions to open subsets. We have decorated them with indices making explicit their source and target, but we have omitted the signs: so, for instance,  we should interpret the diagram  as saying that  we can write
$$
d^n_R|_{H^0(\mathcal{O}_{U_j})} = \pm (d^n_R)^V_{U_j} \pm (d^n_R)^W_{U_j}
$$
up to the appropriate choice of the signs. For more comments  on the issue of signs, see below. 
 The vertical arrows, in black, denote the components of the pull-back $\pi^*$.

Now let us consider an element
$$
\theta = (\theta_1, \theta_2) \in H^0(\mathcal{O}_{\widetilde{U_{j}}}) = H^0(\mathcal{O}_{U}) \oplus  H^0(\mathcal{O}_{U}) 
$$
We  assume that 
$\theta_2=0$. Indeed, we can write $\theta=(\theta_1,0) + (0,\theta_2)$, and our argument carries over without variations to the summand with $\theta_1=0$. We have that 
$$
(d^n_Q)_{\widetilde{U_{j}}}^{\widetilde{V}}(\theta)=\theta_1 \quad  \quad (d^n_Q)_{\widetilde{U_{j}}}^{\widetilde{W}}(\theta)=0
$$
 It is easy to contruct a $\xi' \in H^0(\mathcal{O}_{U_j})$ such that $(d^n_Q)_{U_{j}}^{V}(\xi')=\theta_1$. However  $(d^n_Q)_{U_{j}}^{W}(\xi')$ will not vanish, in general. The two components of $U_j$ are glued along $D$: to construct $\xi'$, we have to extend $\xi$ from $V_j$ to  the whole of $U_j$. The resulting function, in general, will be non-zero over $W_j$. % also on the  second irreducible component of $U_j$. 

However we can define $\xi'$ in such a way that it  is constant away from the $l$-th direction. Let us explain this point more precisely. As the open subset $W_j$ is a product, we have a projection $pr_{W_j}: W_j \to D$. We  can define $\xi'$ as follows
 \begin{itemize}
 \item If $x$ is in $V_j$, then $\xi'(x)=\theta(x)$
 \item If $x$ is in $W_j$, then $\xi'(x)=\theta\circ pr_{W_j}(x)$
 \end{itemize}
As we mentioned, the issue is that 
$$
(d^n_Q)_{U_{j}}^{W}(\xi')(x) =\theta\circ pr(x)
$$
is non-zero in general and so does not match $(d^n_Q)_{\widetilde{U_{j}}}^{\widetilde{W}}(\theta)=0$.

Now we can use the fact that $U_{j''}$ is also a product, and thus it carries a projection $pr_{U_{j''}}:U_{j''} \to D$. We define\footnote{In the definition of $\theta''$ below the choice of the  negative sign is purely indicative. It would be more correct to  leave the sign unspecified: the sign of $\theta''$ actually depends on the signs of the differentials in Lemma \ref{smallcomplex2}, which we chose not to identify as the argument does not depend on them. This sign ambiguity affects also the rest of the proof, e.g.  (\ref{propaga}), but it is harmless.}
\begin{itemize}
\item $
\xi'' \in H^0(\mathcal{O}_{U_{j''}}), \quad \xi''(x)=- \theta\circ pr_{U_{j''}}(x)
$
\item $\xi:=\xi' + \xi'' \in H^0(\mathcal{O}_{U_j}) \oplus H^0(\mathcal{O}_{U_{j''}}) \subset R^{n-1}$
\end{itemize}
We obtain a chain of equalities 
\begin{equation}
\label{propaga}
d^n_R(\xi) = d^n_R(\xi') + d^n_R(\xi'')= ( \theta_1 + \theta\circ pr_{W_j}(x) ) - \theta\circ pr_{W_j}(x) = \theta_1=d^n_Q(\theta)
\end{equation}
The image of $\xi$ under $d^n_R$ matches the image of $d^n_Q(\theta)$, and this concludes the proof.
\end{proof}

\begin{example}
\label{examplep1} 
Let us explain the  proof of Theorem \ref{coho} in the simple case $X=\mathbb{P}^1$. The variety $\cM_X$ has two irreducible components, both isomorphic to $E$, meeting transversely at a point. We denote these two components $E_1$ and $E_2$. 
As top coherent cohomology computes the arithmetic genus, which is invariant in flat families, we already know that 
$$
dim(H^1(\cO_{\cM_X}))=2
$$
Indeed the variety $\cM_X$ admits a flat deformation to a smooth curve of genus $2$. The normalization $\widetilde{\cM_X}$ is the disjoint union of  $E_1$ and $E_2$.

In addition to the low-dimensionality, this example is much simpler than the general case also because, to compute coherent cohomology,  we can directly use the \v{C}ech complex of the cover $\mathfrak{U}_X$. 
We described   $\mathfrak{U}_X$  in Example \ref{procedure}: it contains three open subsets 
\begin{enumerate}
\item $U_1:=E_1-\{e\}$
\item $U_2:=E_2 - \{e\}$
\item and  $U_3$, which is the push-out of the diagram 
$$
E_1-\{p\} \longleftarrow \{e\} \longrightarrow E_2-\{p\}
$$
\end{enumerate}
The normalizations look as follows: $\widetilde{U_1}=U_1$ and  $\widetilde{U_2}=U_2$, while $\widetilde{U_3}=E_1-\{p\} \amalg E_2-\{p\}$. We obtain the following morphism of \v{C}ech complexes 
$$
\xymatrix{
Q^0:=H^0(\cO_{\widetilde{U_1}}) \oplus H^0(\cO_{\widetilde{U_2}}) \oplus H^0(\cO_{\widetilde{U_3}}) \ar[r]^-{d_Q} & Q^1:=H^0(\cO_{E_1-\{e, p\}}) \oplus H^0(\cO_{E_2-\{e, p\}}) \ar[r] & 0 
\\
R^0:=H^0(\cO_{U_1}) \oplus H^0(\cO_{U_2}) \oplus H^0(\cO_{U_3}) \ar[r]^-{d_R}  \ar[u]^{\pi^*} & R^1:=H^0(\cO_{E_1-\{e, p\}}) \oplus H^0(\cO_{E_2-\{e, p\}}) \ar@{=}[u] \ar[r] & 0
}
$$

We want to prove that these two  complexes have the same top  cohomology: equivalently, that $d_R$ and $d_Q$ have the same image in $Q^1=R^1$. Since for $i=1$ and $2$ we have that $U_i = \widetilde{U_i}$, we only need to worry about the image of the summand $H^0(\cO_{\widetilde{U_3}})$. We will check that for every 
$$\theta=(\theta_1, \theta_2) \in H^0(\cO_{\widetilde{U_3}}) \cong H^0((\cO_{E_1-\{p\}})) \oplus 
H^0((\cO_{E_2-\{p\}}))
$$ there is a $\xi \in R^0$ such that $d_Q((\theta_1, \theta_2)) = d_R(\xi)$.

We define a function $\xi'$ on $U_3$ with the following properties: 
\begin{enumerate}
\item $\xi'|_{E_1-\{p\}} = \theta_1$
\item $\xi'|_{E_2-\{p\}} = \theta_2 + ( \theta_1(0) - \theta_2(0))$
\end{enumerate}
This pins down  $\xi'$ uniquely: note that $\xi'$ is well defined, as it takes the same value on $0 \in E_1$ and $0 \in E_2$, which are identified in $\cM_X$. Next we define 
$\xi'' \in H^0(U_2)$ to be the constant function, equal to $\theta_2(0)-\theta_1(0)$. Finally we set 
$$
\xi = (0, \xi'', \xi') \in H^0(\cO_{U_1}) \oplus H^0(\cO_{U_2}) \oplus H^0(\cO_{U_3})
$$
Then we have that 
$$
d_Q(\theta) = - \theta_1|_{E_1-\{p\}} -  \theta_2|_{E_2-\{p\}} = 
- \theta_1|_{E_1-\{p\}} -  \theta_2|_{E_2-\{p\}} - ( \theta_1(0) - \theta_2(0)) + \theta_2(0) - \theta_1(0) = d_R(\xi)
$$
and this concludes the proof.
 \end{example}

 \section{Equivariant elliptic cohomology and derived equivalences}
 \label{derivedequivalences}
    One of the goals of this article is to show that equivariant elliptic cohomology is not invariant under equivariant derived equivalences. This is in contrast with the behaviour of equivariant K-theory and equivariant singular cohomology, which are both derived invariants.  
    
 We will focus on a specific three-dimensional example of good toric varieties that are (equivariantly) derived equivalent, but have different elliptic cohomology. As it will be apparent from our construction, it is easy to produce a great wealth of such examples: in fact, virtually all pairs of equivariantly derived equivalent toric varieties arising in this way  have non-isomorphic elliptic cohomology. The key ingredient is   Kawamata's toric McKay correspondence. We restrict to the simplest setting of finite abelian quotients of $\bA^n$, as this will be sufficient for our purposes, but we point out that Kawamata's results hold in much greater generality.

 Let $G$ be a finite subgroup of the torus $T$ acting on $\bA^n$. Assume that $G$ is a subgroup of $SL_n(\bC)$, that is, the $G$-action preserves the Calabi--Yau structure on  $\bA^n$ given by
 $$
 \Omega = dx_1 \wedge \ldots \wedge dx_n
 $$
 The quotient $X=\bA^n/G$ is a singular simplicial toric variety. Up to isomorphisms of $N_T$, the fan of $X$ has the following shape:
choose a splitting $N_T \cong H \otimes \mathbb{Z}$, where $H$ is a lattice of dimension $n-1$; then there is a lattice simplex $\Delta \subset H \otimes \mathbb{R}$ such that the fan of $X$ is the cone over $\Delta$. Namely   $\Delta$ carries a natural stratification whose strata are subsimplices $\Delta_\tau \subset \Delta$. Then the set 
 $$
 \tau := \{v \in N_T \otimes \mathbb{R} \, \, | \, \, v = \lambda x, \quad x \in \Delta_{\tau}, \, \lambda \in \mathbb{R}_{>0} \}
 $$
 is a closed convex cone, and 
 $$
 \Sigma_X = \{ \tau \, \, | \, \, \Delta_\tau \subset \Delta \}
 $$
% Note that $\Sigma_X$ has only one top-dimensional cone, corresponding to its unique torus fixed point. 

 \emph{Crepant} toric resolutions of $X$ are in bijection with unimodular triangulations of $\Delta$;  here crepant means that the resolution is also CY, i.e. it has a trivializable canonical bundle. Given one such triangulation $\cT$, the cones over the strata of $\cT$ form a smooth fan $\Sigma_\cT$. We set $X_\cT:=X_{\Sigma_\cT}$. As $\Sigma_\cT$ is a subdivision of $\Sigma$ it defines a toric map
 $ \, 
 p_\cT: X_\cT \to X
 $, 
 which is a resolution of singularities. 
 
 \begin{example}
 Let us consider the two-dimensional case of the story described above. We set  $T:=(\mathbb{C}^*)^2$, and $N_T:=\mathbb{Z}^2$. Consider $\bA^2$ with the standard $T$-action. We choose coordinates $(h_1, h_2)$ on $N_T$ and 
 $(x_1, x_2)$ on $\mathbb{A}^2$. Let $G$ be the group $\mu_n$ of $n$-th roots of unity acting anti-diagonally on $\mathbb{A}^2$
 $$
 \xi \in \mu_n, \quad \xi \cdot (x_1, x_2) = (\xi x_1, \xi^{-1} x_2)
 $$ 
 This action preserves the CY form $\omega=dx_1\wedge dx_2$; it is easy to see that these are the only finite subgroups of $T$ preserving it. Up to isomorphisms of $N_T$, the fan of 
 $\mathbb{A}^n/G$ can be described as follows.

 Consider  the line 
 $H \subset N_T$ given by $H=\{(h_1, h_2)\, | \, h_2=1\}.$ We  equip $H$ with the  coordinate $h_1$. 
 A lattice polytope in $H \otimes \mathbb{R}$ is just an interval, and we let $$\Delta =\{ h_1 \in H \, | \, 0 \leq h_1 \leq n \}$$ The fan of $\mathbb{A}^n/G$ is given by the cone over the strata of $\Delta$. Unimodular triangulations here are just subdivisions in subintervals of length one, and there is evidently a unique such triangulation $\cT$ of $\Delta$. The toric surface $X_\cT \to X$ is an iterated toric blow-up of $X$, and is a crepant resolution of $X$.
 \end{example}
       
There is one more object playing a role in the McKay correspondence, namely the stacky quotient $Y:=[\mathbb{A}^n/G]$. This is smooth a toric DM stack, in the sense of \cite{fantechi2010smooth}. The coarse moduli map 
 $q: [\mathbb{A}^n/G] \to \mathbb{A}^n/G$ is toric and birational, so it can be viewed as a stacky resolution. %The stacky quotient $[\mathbb{A}^n/G]$  can also be viewed as a toric resolution of $X$, since $[\mathbb{A}^n/G]$ is a smooth stack. 
\begin{proposition}
\label{equivderiv}
Let $G$, $X$ $\Delta$ and $Y$ be as above. Let $\cT$ and $\cT'$ be two unimodular triangulations of $\Delta$. 
\begin{enumerate}
\item There is an equivalence 
     $$
\Perf([Y/T]) \simeq \Perf([X_{\cT}/T])
     $$  
     which intertwines the natural monoidal action  of $\Perf([*/T])$ on its source and   target 
     \item There is an equivalence 
     $$
\Perf([X_{\cT}/T]) \simeq \Perf([X_{\cT'}/T])
     $$  
     which intertwines the natural monoidal action  of $\Perf([*/T])$ on its source and   target 
     \end{enumerate}
       \end{proposition}
       \begin{proof}
       Let us start recalling a few well-known   facts. If $S$ is a toric stack then $\Perf(S)$ carries an action of $T$ given by pull-back along the action on $S$. Since we are working with $\infty$-categories we can meaningfully take  $T$-invariants,  and we find that
       $$
       \Perf(S)^T \simeq \Perf([S/T])
       $$       
   Next, let $S$ and $S'$ be toric DM stacks and let $f:S \to S'$ be a proper toric map. Then the pull-back and the push-forward along $f$
          $$
       f^*: \Perf(S') \to \Perf(S) \quad f_*:\Perf(S) \to \Perf(S')
       $$
       are compatible with the torus action on their source and target. Thus, taking   $T$-invariants we obtain functors  
$$
       f^*: \Perf([S'/T]) \to \Perf([S/T]) \quad f_*:\Perf([S/T]) \to \Perf([S'/T])
       $$
       that intertwine the monoidal action of $\Perf([*/T])$ on their source and target. 
       
       Now the first claim follows from Kawamata's toric McKay correspondence \cite[Theorem 4.2]{kawamata2005}. Kawamata constructs an explicit derived equivalence between $Y$ and $X_\cT$. Namely, let $\cX$ be the normalization of the fiber product $Y \times_X X_\cT$. Then $\cX$ is a toric stack equipped with toric maps
       $$
       \xymatrix{ & \cX \ar[dr]^p \ar[dl]_q&\\
       Y && X_\cT}
       $$
       Kawamata proves that the composition 
       $
       p_*q^*: \Perf(Y) \to \Perf(X)
       $
       is an equivalence; further, it is  compatible with the $T$-action on its source and target, since $p$ and $q$ are toric. Taking $T$-invariants proves the first claim. The second claim follows from the first, because  $X_\cT$ and $X_{\cT'}$ are both equivariantly derived equivalent to $Y$.
       \end{proof}
   
\begin{example}
\label{examplenoniso}
We are now ready to present a simple example of good toric 3-folds, which are equivariantly derived equivalent, but have non isomorphic equivariant elliptic cohomology. Consider the action of the standard torus $T=(\mathbb{C}^*)^3$ on $\mathbb{A}^3$. Let $G$ be the kernel of the  homomorphism of groups
$$
m: (\mu_2)^3  \to \mu_2,  \quad m(\xi_1, \xi_2, \xi_3)=\xi_1\xi_2\xi_3
$$  
The group $G$ is a subgroup of $(\mathbb{C}^*)^3$ and is contained in $SL_3(\mathbb{C})$. The fan of $X=\mathbb{A}^3/G$ is the cone over the  2-dimensional lattice polytope $\Delta$ below
$$
    
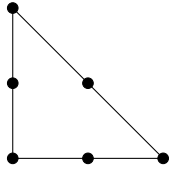
\begin{tikzpicture} 
\filldraw [black] (3,0) circle (2pt); 
\filldraw [black] (3,1) circle (2pt); 
\filldraw [black] (3,2) circle (2pt); 
\filldraw [black] (5,0) circle (2pt); 
\filldraw [black] (4,1) circle (2pt); 
\filldraw [black] (4,0) circle (2pt); 
\draw (3,0) -- (3,2); 
\draw (3,0) -- (5,0); 
\draw (5,0) -- (3,2); 
%\draw (3,1) -- (4,1); 
%\draw (3,1) -- (4,0); 
%\draw (4,0) -- (4,1); 

%\node at (3.5,0.25) {$\sigma_1$}; 
%\node at (3.75,0.75) {$\sigma_2$}; 
%\node at (4.5,0.25) {$\sigma_3$}; 
%\node at (3.5,1.25) {$\sigma_4$}; 
%
%
%
%\filldraw [black] (10,0) circle (2pt); 
%\filldraw [black] (10,1) circle (2pt); 
%\filldraw [black] (10,2) circle (2pt); 
%\filldraw [black] (12,0) circle (2pt); 
%\filldraw [black] (11,1) circle (2pt); 
%\filldraw [black] (11,0) circle (2pt); 
%\draw (10,0) -- (10,2); 
%\draw (10,0) -- (12,0); 
%\draw (12,0) -- (10,2); 
%\draw (11,0) -- (11,1); 
%\draw (10,1) -- (11,0); 
%\draw (11,0) -- (10,2); 
%
%
%\node at (10.5,0.25) {$\sigma_1$}; 
%\node at (10.5,0.7) {$\sigma_2$}; 
%\node at (11.5,0.25) {$\sigma_4$}; 
%\node at (10.75,1) {$\sigma_3$}; 
\end{tikzpicture}
   $$
Consider the following two unimodular triangulations $\cT$ and $\cT'$ of $\Delta$
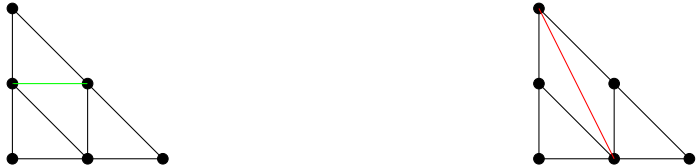
$$
    \begin{tikzpicture} 
\filldraw [black] (3,0) circle (2pt); 
\filldraw [black] (3,1) circle (2pt); 
\filldraw [black] (3,2) circle (2pt); 
\filldraw [black] (5,0) circle (2pt); 
\filldraw [black] (4,1) circle (2pt); 
\filldraw [black] (4,0) circle (2pt); 
\draw (3,0) -- (3,2); 
\draw (3,0) -- (5,0); 
\draw (5,0) -- (3,2); 
\draw[green] (3,1) -- (4,1); 
\draw (3,1) -- (4,0); 
\draw (4,0) -- (4,1);

\filldraw [black] (10,0) circle (2pt); 
\filldraw [black] (10,1) circle (2pt); 
\filldraw [black] (10,2) circle (2pt); 
\filldraw [black] (12,0) circle (2pt); 
\filldraw [black] (11,1) circle (2pt); 
\filldraw [black] (11,0) circle (2pt); 
\draw (10,0) -- (10,2); 
\draw (10,0) -- (12,0); 
\draw (12,0) -- (10,2); 
\draw (11,0) -- (11,1); 
\draw  (10,1) -- (11,0); 
\draw[red] (11,0) -- (10,2); 
\end{tikzpicture}
   $$
The toric 3-folds $X_\cT$ and $X_{\cT'}$ are crepant resolutions of $X$. They are related by an Atiyah flop. By Proposition \ref{equivderiv} there is an equivalence 
$$
\Perf([X_\cT/T]) \simeq \Perf([X_{\cT'}/T])
$$
which intertwines the monoidal $\Perf([*/T])$-action on source and target. Now consider the arcs that have been flopped. The green arc gives rise to a cone $\tau \in (\Sigma_{X_\cT}^{2})^*$ having the property that there is no $\tau' \in (\Sigma_{X_{\cT'}}^{2})^*$ such that 
$\langle \tau \rangle = \langle \tau' \rangle$. By Theorem \ref{main1}, $\Ell_T(X_\cT)$ is not isomorphic to $\Ell_T(X_{\cT'})$. 
\end{example}
\begin{remark}
Distinct toric crepant resolutions of the same simplicial quotient singularity will often (always?) have non-isomorphic equivariant elliptic cohomology; though, by Proposition \ref{equivderiv} they are equivariantly derived equivalent. Example \ref{examplenoniso} is only the simplest instance where this phenomenon occurs. Using our  Theorem \ref{main1} we can check that this happens for most cases, in all dimensions.
%provides an easily computable test, which is sufficient in most cases to check that equivariant elliptic cohomologies are non-isomorphic. 

  In particular, although Example \ref{examplenoniso} is about a pair of non-proper varieties, it is  easy to generate examples which are proper. For instance, since $G$ is a subgroup of $T$, if we  compactify $\mathbb{A}^n$ to $\mathbb{P}^n$, the $G$-action will naturally extend to $\mathbb{P}^n$.  Costructing a toric crepant resolution of $\mathbb{P}^n/G$ amounts to choosing in a compatible way resolutions of the toric affine patches around each torus fixed point. For an illustration of  the beautiful geometry of these resolutions we can refer the reader, for instance, to Figure 1 of \cite{pascaleff2022singularity}. Different choices will give rise to equivariantly derived equivalent toric varieties, by gluing the local equivalences from  Proposition \ref{equivderiv}.   For many of these  resolutions, one can verify that equivariant elliptic cohomologies are different by applying Theorem \ref{main1}, exactly as we did in Example \ref{examplenoniso}. \end{remark}

\begin{remark}
Totaro  \cite{totaro2000chern} has shown that the \emph{elliptic genus} is invariant under flops, and in fact is universal with that property. This implies in particular that the examples coming from crepant resolutions in the toric setting, which we have considered in this article, will have different (equivariant) elliptic cohomology  but the   same elliptic genus. This suggests the very interesting question whether the elliptic genus is an invariant of the derived category,  although elliptic cohomology  is not.
 %From this result follows that the elliptic genera of the smooth Calabi-Yau toric varities $X_1$ and $X_2$ are indeed isomorphic even though their elliptic cohomology is not. 
\end{remark}

%\bibliographystyle{alphamod}
%\bibliography{chern}

\begin{thebibliography}{GKM98}
\providecommand{\url}[1]{\href{#1}{{\def~{\textasciitilde}\tt #1}}}

\bibitem[AS69]{atiyah1969equivariant}
M.~F. Atiyah and G.~B. Segal, \emph{Equivariant $ K $-theory and completion},
  Journal of Differential Geometry \textbf{3} (1969), no.~1-2, pp.~1--18

\bibitem[And03]{ando2003sigma}
M.~Ando, \emph{The sigma orientation for analytic circle-equivariant elliptic
  cohomology}, Geometry \& Topology \textbf{7} (2003), no.~1, pp.~91--153

\bibitem[BD04]{baas2004elliptic}
N.~A. Baas and B.~I. Dundas, \emph{of elliptic cohomology}, Topology, Geometry
  and Quantum Field Theory: Proceedings of the 2002 Oxford Symposium in Honour
  of the 60th Birthday of Graeme Segal, no. 308, Cambridge University Press,
  2004, p.~18

\bibitem[BGT13]{blumberg2013universal}
A.~J. Blumberg, D.~Gepner, and G.~Tabuada, \emph{A universal characterization
  of higher algebraic K-theory}, Geometry \& Topology \textbf{17} (2013),
  no.~2, pp.~733--838

\bibitem[Bak10]{bakhtary2010cohomology}
P.~Bakhtary, \emph{On the cohomology of a simple normal crossings divisor},
  Journal of Algebra \textbf{324} (2010), no.~10, pp.~2676--2691

\bibitem[Bla16]{blanc2016topological}
A.~Blanc, \emph{Topological K-theory of complex noncommutative spaces},
  Compositio Mathematica \textbf{152} (2016), no.~3, pp.~489--555

\bibitem[FMN10]{fantechi2010smooth}
B.~Fantechi, E.~Mann, and F.~Nironi, \emph{Smooth toric Deligne-Mumford stacks}

\bibitem[Ful93]{fulton1993introduction}
W.~Fulton, \emph{Introduction to toric varieties}, no. 131, Princeton
  university press, 1993

\bibitem[GKM98]{goresky1998equivariant}
M.~Goresky, R.~Kottwitz, and R.~MacPherson, \emph{Equivariant cohomology,
  Koszul duality, and the localization theorem}, Inventiones mathematicae
  \textbf{131} (1998), no.~1, pp.~25--84

\bibitem[GR19]{gaitsgory2019study}
D.~Gaitsgory and N.~Rozenblyum, \emph{A study in derived algebraic geometry:
  Volume I: correspondences and duality}, vol. 221, American Mathematical
  Society, 2019

\bibitem[Gan14]{ganter2014elliptic}
N.~Ganter, \emph{The elliptic Weyl character formula}, Compositio Mathematica
  \textbf{150} (2014), no.~7, pp.~1196--1234

\bibitem[Gro94]{grojnowski1994delocalised}
I.~Grojnowski, \emph{Delocalised equivariant elliptic cohomology}, Elliptic
  Cohomology, London Math. Soc. Lecture Note Ser \textbf{342} (1994),
  pp.~114--121

\bibitem[HLP20]{halpern2020equivariant}
D.~Halpern-Leistner and D.~Pomerleano, \emph{Equivariant Hodge theory and
  noncommutative geometry}, Geometry \& Topology \textbf{24} (2020), no.~5,
  pp.~2361--2433

\bibitem[Kaw05]{kawamata2005}
Y.~Kawamata, \emph{Log Crepant Birational Maps and Derived Categories}, Journal
  of Mathematical Sciences-University of Tokyo \textbf{12} (2005), no.~2,
  pp.~211--232

\bibitem[Lur09]{lurie2009survey}
J.~Lurie, \emph{A survey of elliptic cohomology}, Algebraic topology, Springer,
  2009, pp.~219--277

\bibitem[Mas08]{masuda2008equivariant}
M.~Masuda, \emph{Equivariant cohomology distinguishes toric manifolds},
  Advances in Mathematics \textbf{218} (2008), no.~6, pp.~2005--2012

\bibitem[PS22]{pascaleff2022singularity}
J.~Pascaleff and N.~Sibilla, \emph{Singularity categories of normal crossings
  surfaces, descent, and mirror symmetry}, arXiv preprint arXiv:2208.03896
  (2022)

\bibitem[Rob15]{robalo2015k}
M.~Robalo, \emph{K-theory and the bridge from motives to noncommutative
  motives}, Advances in Mathematics \textbf{269} (2015), pp.~399--550

\bibitem[Ros03]{rosu2003equivariant}
I.~Rosu, \emph{Equivariant K-theory and equivariant cohomology}, Mathematische
  Zeitschrift \textbf{243} (2003), no.~3, pp.~423--448

\bibitem[Sch05]{schwede2005gluing}
K.~Schwede, \emph{Gluing schemes and a scheme without closed points},
  Contemporary Mathematics \textbf{386} (2005), p.~157

\bibitem[Tot00]{totaro2000chern}
B.~Totaro, \emph{Chern numbers for singular varieties and elliptic homology},
  Annals of Mathematics (2000), pp.~757--791

\bibitem[hg]{68568}
A.~G. (https://mathoverflow.net/users/1/anton geraschenko), \emph{Is there
  always a toric isomorphism between isomorphic toric varieties?},
  MathOverflow, URL:https://mathoverflow.net/q/68568 (version: 2011-06-23)

\end{thebibliography}

\end{document}